\documentclass[final,leqno,letterpaper]{etna}
\usepackage{todonotes}
\usepackage{amsmath}  
\usepackage{amsfonts} 
\usepackage{amssymb}

\usepackage[nameinlink,capitalize]{cleveref}
\setbibdata{1}{xx}{46}{2017} 

\hypersetup{%
    pdftitle={Convergence Rates for Oversmoothing Banach Space Regularization},
    pdfauthor={Philip Miller, Thorsten Hohage},
    pdfkeywords={regularization, convergence rates, oversmoothing, BV-regualarization, 
sparsity promoting wavelet regularization, statistical inverse problems}
    }

\title{Convergence Rates for Oversmoothing Banach Space Regularization\thanks{%
This work has been supported by Deutsche Forschungsgemeinschaft (German Research Foundation, DFG) through Grant RTG 2088, project B01.
}}

\author{Philip Miller\and Thorsten Hohage\footnotemark[2]}

\shorttitle{Convergence rates for oversmoothing} 
\shortauthor{P.~MILLER AND T.~HOHAGE}

\newtheorem{assumption}[theorem]{Assumption}

\DeclareMathOperator*{\argmin}{argmin}

\newcommand{\inv}{^{-1}}             
\newcommand{\gobs}{g^\mathrm{obs}}

\newcommand{\nat}{\mathbb{N}_0}

\newcommand{\X}{\mathbb{X}}
\newcommand{\Y}{\mathbb{Y}}

\newcommand{\infbracket}[1]{\left[#1\right]}
\newcommand{\xh}{\hat{x}_\alpha}

\newcommand{\wav}{\mathcal{S}}
\newcommand{\bspace}[3]{b^{#1}_{{#2},{#3}}}
\newcommand{\Bspace}[3]{B^{#1}_{{#2},{#3}}(\Omega)}
\newcommand{\tildeB}[3]{\tilde{B}^{#1}_{{#2},{#3}}(\Omega)}
\newcommand{\Bospace}[3]{B^{#1}_{{#2},{#3}}}
\newcommand{\Bn}[4]{\left\| {#1} \right\|_{B^{#2}_{{#3},{#4}}}}
\newcommand{\bn}[4]{\left\| {#1} \right\|_{{#2},{#3},{#4}}}

\newcommand{\BV}{\mathrm{BV}}
\newcommand{\XR}{{\X_\mathcal{R}}}
\newcommand{\XA}{{\X_-}}
\newcommand{\XL}{{\X_{\mathrm{L}}}}
\newcommand{\fh}{\hat{f}_\alpha}
\newcommand{\gh}{\hat{g}_\alpha}
\newcommand{\Brspace}[3]{B^{#1}_{{#2},{#3}}(\mathbb{R}^d)}

\begin{document}

\maketitle

\renewcommand{\thefootnote}{\fnsymbol{footnote}}

\footnotetext[2]{Institute for Numerical and Applied Mathematics, University of G\"ottingen, Germany}

\begin{abstract}
This paper studies Tikhonov regularization for finitely smoothing operators in Banach spaces when the penalization 
enforces too much smoothness in the sense that the penalty term is not finite at the true 
solution. In a Hilbert space setting, Natterer (1984) showed with the help of spectral theory that optimal rates can be achieved in this situation. ('Oversmoothing does not harm.') 
For oversmoothing variational regularization in Banach spaces only very recently progress 
has been achieved in several papers on different settings, all of which construct families of smooth approximations to the true solution. In this paper we propose to construct such a family 
of smooth approximations based on $K$-interpolation theory. We demonstrate that this 
leads to simple, self-contained proofs and to rather general results. 
In particular, we obtain optimal convergence 
rates for bounded variation regularization, general Besov penalty terms and 
$\ell^p$ wavelet penalization with $p<1$ which cannot 
be treated by previous approaches. We also derive minimax optimal rates for white noise 
models. Our theoretical results are confirmed in numerical experiments. 
\end{abstract}



\section{Introduction} Inverse problems occur in many areas of science and engineering when a quantity of interest $f$
is not directly accessibly, and only indirect effects $\gobs$ can be observed under noise. 
Very often such inverse problems are formulated in the form of operator equations 
\[
F(f) = g
\]
with some injective, but possibly nonlinear forward operator 
$F:D_F\to \Y$ mapping a subset $D_F$ of some Banach space to another Banach space $\Y$. 
Typically these operator equations are ill-posed in the sense that the inverse of $F$ fails 
to be continuous with respect to useful Banach norms. Such problems have been studied 
in numerous papers and monographs, we only refer to \cite{EHN:96,Scherzer_etal:09,Schuster2012}.

The probably most common and well-known method to deal with ill-posedness for inexact 
observed data $\gobs$ and compute stable reconstructions of $f$ is Tikhonov regularization.
If $\gobs$ belongs to $\Y$ with deterministic error bound
\begin{align}\label{eq:det_noise_level}
\|\gobs-F(f)\|_{\Y}\leq \delta,
\end{align}
we consider Tikhonov regularization in the form 
\begin{align}\label{eq:Tik_det}
S_\alpha (\gobs) &  := \argmin_{h \in  D_F \cap \XR} \, \infbracket{ \frac{1}{2 \alpha} \|\gobs- F(h)  \|_{\Y}^2  + \frac{1}{u}\|h\|_{\X_\mathcal{R}}^u} 
\end{align}
with a penalty term  $\frac{1}{u}\|h\|_{\X_\mathcal{R}}^u$ given by the norm of 
another Banach space $\X_{\mathcal{R}}$,  
a regularization parameter  $\alpha>0$, and an exponent $u\in (0,\infty)$.
Later in \Cref{sec:white_noise} we will also consider a variant of \eqref{eq:Tik_det} for  
a white noise model. 

Oversmoothing refers to the situation that the true solution $f$ does not belong to 
the space $\X_{\mathcal{R}}$. This situation is likely to occur if the norm of 
$\X_{\mathcal{R}}$ contains derivatives. The use of such penalty terms is common practice 
and was already proposed in the original paper by Tikhonov \cite{tikhonov:63}. 
As usual in regularization theory, we aim to bound the reconstruction error in terms of the 
noise level $\delta$. To give a specific example, we may be interested to bound the 
$L^p$-error for image deblurring with bounded variation regularization in the presence 
of texture. 

As the Tikhonov reconstructions in \eqref{eq:Tik_det} belong to $\X_\mathcal{R}$, but $f\notin \X_{\mathcal{R}}$, one cannot expect them to converge to 
$f$ in $\X_\mathcal{R}$. Instead, the reconstruction error will be measured in a 
weaker norm $\|\cdot\|_\XL$ indexed by the subscript $L$ for 'loss function'. 
We will further assume that $F$ is finitely smoothing in the sense that is satisfies 
a two-sided Lipschitz condition with respect to the norm of an even large space $\X_-$
(typically with negative smoothness index), and $\XL$ contains or coincides with 
some real interpolation space between  $\X_-$ and $\X_{\mathcal{R}}$. 

Let us briefly sketch the literature on oversmoothing regularization: 
In a first seminal paper \cite{natterer:84} inspiring numerous follow-up works, 
Natterer analyzed the case that 
$\Y$ is a Hilbert space, $F$ is linear, $u=2$, and $\X_\mathcal{R}$, $\XL$ and $\X_-$ all belong 
to a Hilbert scale, using the Heinz inequality for self-adjoint operators in 
Hilbert spaces as a main tool. In Banach space settings, only variational techniques 
are available, and usually a first step is to derive an inequality for the Tikhonov 
estimator by plugging the true solution $f$ into the Tikhonov functional. 
In the oversmoothing case this is not possible, 
and only recently progress has been achieve for this situation by several constructions 
of sequences of smooth elements approximating the true solution $f$: 
Hofmann \& Math\'e \cite{HM:18} (see also  \cite{HP:20}) consider nonlinear operators, still in Hilbert spaces, 
but their approach for constructing smooth approximations to $f$ (auxiliary elements in their terminology) is already essentially a special case of our approach. 
In \cite{GH:19} the case of
$\ell^1$-regularization with $\ell^2$-loss and a diagonal operator was studied 
using truncation of the sequence $f$. In a previous work \cite{MH:20} 
the authors analyzed oversmoothing 
in sparsity promoting wavelet regularization using hard thresholding to approximate $f$ 
by smooth elements. The most general results so far have been obtained by Chen, Hofmann \& Yousept \cite{CHY:21} who use functional calculus of sectorial operators  to construct smooth approximating sequences. 

In this paper we propose to construct a sequences of smooth approximations to $f$ 
based on $K$-interpolation theory. We believe that our analysis is 
significantly simpler than the one in \cite{CHY:21}. Moreover, we can derive optimal 
rates for some interesting cases such as $\BV$-regularization and Besov-space regularization 
with $p=1$ that do not seem to be covered by the analysis in \cite{CHY:21}. 

We also derive convergence rates for oversmoothing regularization with statistical 
noise models covering both Besov space and $\BV$ regularization. 
It seems that oversmoothing for statistical inverse problems has not received 
much attention in the literature so far, we are only aware of the 
preprint \cite{rastogi:20}. 

The remainder of this paper is organized as follows: In the following 
\Cref{sec:det_analysis} we introduce our setting and  
prove our main result (Theorem \ref{thm:rate_abstract}) for the deterministic noise model 
\eqref{eq:det_noise_level}. 
In \Cref{sec:besov} we formulate and discuss a convergence rate theorem for general 
oversmoothing Besov space regularization as a corollary to Theorem \ref{thm:rate_abstract}. 
In the following \Cref{sec:bv} we show $L^p$ error bounds for oversmoothing 
bounded variation regularization in a further corollary to Theorem \ref{thm:rate_abstract}. 
Oversmoothing regularization for statistical inverse problems is treated in 
\Cref{sec:white_noise} by 
adapting the proof of Theorem \ref{thm:rate_abstract}. 
We also discuss a parameter identification problem for an elliptic differential equation 
as a specific example and confirm the predicted convergence rates for this example in numerical experiments. 
The paper finishes with some conclusions and three appendices collecting results on  
interpolation theory, Besov spaces, and functions of bounded variation. 

\section{Deterministic analysis} \label{sec:det_analysis}
In this section we present our main result. We will assume that $\XR$ is a quasi-Banach space.
Recall that a quasi-Banach space $\X$ with norm $\|\cdot\|_{\X}$ satisfies all axioms 
of a Banach space except for the triangle inquality, which only holds true in the 
weaker form $\|x+y\|_{\X}\leq c_{\X}(\|x\|_{\X}+\|y\|_{\X})$ with some constant $c_{\X}\geq 1$ 
independent of $x,y\in\X$. The most prominent examples of quasi-Banach spaces that are 
not Banach spaces are $L^p$ and $\ell^p$ spaces with $p\in (0,1)$. 
$\ell^p$ penalty terms with $p\in (0,1)$ have been proposed by a number of authors 
(see, e.g.\ \cite{BL:09,
SB:14,zarzer:09}) 
with the aim to enforce more sparsity of the regularizers, and this is our reason for 
not confining ourselves to a Banach space penalties. Quasi-Banach space penalties to not cause 
any additional complications in our analysis and may thus be considered the natural 
setting for our approach. 

\subsection{Real interpolation of quasi-Banach spaces}\label{sec:Kinterpol} 
Our analysis is based on real interpolation theory of quasi-Banach spaces via the $K$-method which we will recall in the following.

Let $\X$ and $\XA$ be quasi-Banach spaces with a continuous embedding $\X\subset \XA$. 
The $K$-functional is given by
\begin{align}\label{eq:K_functional_embedded}
 K(t,f)= \inf_{h\in \X}\infbracket{\| f-h \|_\XA + t \| h\|_\X } \quad\text{for }  t>0 \text{ and } f\in \XA.
\end{align}
With this a scale of quasi-norms is defined by 
\[ \|f\|_{\left(\XA,\X\right)_{\theta,q}} = \left(\int_0^\infty \left(t^{-\theta} K(t,f)\right)^q\, \frac{ \mathrm{d} t}{t} \right)^\frac{1}{q}\] 
for  $0<\theta<1$ and $q\in [1,\infty)$ and 
\[ \|f\|_{\left(\XA,\X\right)_{\theta,\infty}} = \sup_{t>0} t^{-\theta} K(t,f) \] 
for $0\leq \theta \leq 1$. 
We obtain quasi-Banach spaces $\left( \XA, \X\right)_{\theta,q}$ consisting of all $f\in \XA$ with $\|f\|_{\left( \XA, \X\right)_{\theta,q}} <\infty$
 (see e.g. \cite[Sec.~ 3.11.]{Bergh1976}). 

\subsection{Assumptions and preliminaries} 
Our basic assumption on the forward operator $F$ is a two-sided Lipschitz condition with 
respect to the norm in $\X_-$. Similar conditions have been imposed in all previous papers 
on oversmoothing Tikhonov regularization that we are aware of. 
We start with $F$ defined on $\tilde{D}_F:= D_F\cap \X_\mathcal{R}$. 

\begin{assumption}\label{ass:oversmooth} 
Suppose $\X_\mathcal{R}$ is a quasi-Banach space and $\Y$ is a Banach space, $\tilde{D}_F\subset \X_\mathcal{R}$ and $F\colon \tilde{D}_F\rightarrow \Y$ a map.
Moreover, we assume that $\X_\mathcal{R}$ continuously embeds into a Banach space $\X_-$ with 
\[\frac{1}{M_1} \|f_1-f_2\|_{\X_-}\leq \|F(f_1)-F(f_2)\|_\Y \leq M_2 \|f_1-f_2\|_{\X_-} \quad\text {for all } f_1,f_2 \in \tilde{D}_F  \]
for some constants $M_1,M_2>0$. 
Finally, let $\xi\in [0,1)$. If $\xi\in (0,1)$, let $\XL$ be a Banach space and suppose that there exists a continuous embedding \[ (\XA,\XR)_{\xi,1}\subset \XL.\] If $\xi=0$, we set $\XL:=\XA.$
\end{assumption}

Note that under this assumption $F$ has a unique continuous extension again denoted by $F$ to the norm closure $D_F$ of $\tilde{D}_F$ in $\XA$.\\
We start with a lemma that introduces smooth approximations to $f$ 
based on real interpolation theory and provides estimates of their approximation rates 
in $\X_-$ and $\XL$ and their growth rate in $\XR$. 
\begin{lemma}[smooth approximations]\label{lem:aux_elements}
Suppose Assumption \ref{ass:oversmooth} holds true.
Let $\theta\in (\xi,1]$ and $\varrho>0$. Suppose  $f \in \left(\X_-,\XR\right)_{\theta,\infty}$ with ${\|f\|_{\left(\X_-, \XR\right)_{\theta,\infty}}\leq \varrho}$. Then there exists a net $(f_t)_{t>0}\subset \XR$ such that the following bounds hold true:
\begin{subequations}
\begin{align}\label{eq:aux_image}
\|f-f_t\|_{\XA}&\leq 2 \varrho t^\theta \\ 
\label{eq:aux_loss}
\|f-f_t \|_{\XL}&\leq  C_L \varrho t^{\theta-\xi}\\
\label{eq:aux_penalty}
\|f_t\|_\XR & \leq 2 \varrho t^{\theta-1}
\end{align}
\end{subequations}
Here $C_L>0$ denotes a constant that is independent of $\varrho,t$ and $f.$
\end{lemma}
\begin{proof}
Recall that ${\|f\|_{\left(\X_-, \XR\right)_{\theta,\infty}}\leq \varrho}$ implies $K(t,f)\leq \varrho  t^\theta$ for $t>0$ with the $K$-functional from \eqref{eq:K_functional_embedded}.
Hence, for every $t>0$ there exists $f_t\in \XR$ such that  
\[   \|f-f_t\|_{\X_-} + t \|f_t \|_\XR  \leq 2 K(t,f)\leq 2 \varrho t^\theta. \]
We neglect the first summand on the left hand side to see \eqref{eq:aux_penalty} and the second to obtain \eqref{eq:aux_image}. This finishes the proof for $\xi=0$, and we now turn to the case $\xi\in (0,1).$\\ 
As an intermediate step to \eqref{eq:aux_loss} we first prove that $\|f-f_t\|_{\left(\X_-, \XR\right)_{\theta,\infty}}\leq 3 \varrho$ for all $t>0$.  To this end we first consider $s\geq t$ and insert $h=0$ into the $K$-functional to wind up with 
\[ K(s, f-f_t )=\inf_{h\in \XR} \infbracket{ \|f-f_t-h\|_{\X_-} + s \|h\|_\XR }
\leq \|f-f_t\|_{\X_-}\leq 2\varrho t^\theta \leq 2\varrho s^\theta.
\] 
For $s\leq t$ we substitute $h=h^\prime -f_t$ and use the triangle inequality in $\XR$ to estimate 
\begin{align*}
K(s, f-f_t ) & = \inf_{h^\prime\in \XR}\infbracket{ \|f-h^\prime\|_{\X_-} + s \|h^\prime - f_t\|_\XR } \\ 
& \leq K(s,f)+ s \|f_t\|_\XR    
\leq \varrho s^\theta + 2 \varrho s t^{\theta-1}
\leq 3 \varrho s^\theta. 
\end{align*}
From the last two inequalities we conclude that 
\[ \|f-f_t\|_{\left(\X_-, \XR\right)_{\theta,\infty}}=\sup_{s>0} s^{-\theta} K( f-f_t,s ) \leq  3 \varrho.\]
By the reiteration theorem (see Proposition \ref{prop:reiteration}) we have
\begin{align*}
  \XL\supset   \left(\XA,\XR \right)_{\xi,1} =\left( \XA , \left(\XA,\XR \right)_{\theta,\infty}\right)_{\frac{\xi}{\theta},1}
\end{align*}
with equivalent norms of the latter two spaces. Hence, Lemma \ref{lem:intrerpolation_inequality} provides an interpolation inequality  $\|\cdot \|_{\XL}\leq c \|\cdot \|_{\XA}^{1-\frac{\xi}{\theta}}\cdot  \| \cdot \|_{\left(\XA,\XR \right)_{\theta,\infty}}^\frac{\xi}{\theta}$.
Inserting $f-f_t$ we finally get
\[  \|f-f_t \|_{\XL} \leq c \left(2\varrho t^\theta\right)^{1-\frac{\xi}{\theta}} \left(3\varrho\right)^\frac{\xi}{\theta}\leq 3c\varrho t^{\theta-\xi}.
\]
\end{proof}
\begin{remark} \label{rem:aux_converse}
From the existence of approximations as in Lemma \ref{lem:aux_elements} one can reclaim the regularity assumption as follows:
Let $f\in \XA$ and suppose that there exists a net $(f_t)_{t>0}\subset \XR$ such that the bounds \eqref{eq:aux_image} and \eqref{eq:aux_penalty} hold true. Inserting $f_t$ for $h$ in the $K$ functional yields $K(t,f)\leq 4 \varrho t^\theta$. Hence $f \in \left(\X_-,\XR\right)_{\theta,\infty}$ with ${\|f\|_{\left(\X_-, \XR\right)_{\theta,\infty}}\leq 4\varrho}$.
\end{remark}
\subsection{Abstract convergence rate result}
With the Lemma \ref{lem:aux_elements} at hand we are in position to prove the following convergence estimates as main result of this paper: 
\begin{theorem}[Error bounds]\label{thm:rate_abstract}
Suppose Assumption \ref{ass:oversmooth} holds true. Let $\theta\in (\xi, 1]$ and $\varrho>0$ . Assume that $f \in \left(\X_-,\XR\right)_{\theta,\infty}$ with ${\|f\|_{\left(\X_-, \XR\right)_{\theta,\infty}}\leq \varrho}$ and moreover that $D_F$ contains an $\XL$-ball with radius $\tau>0$ around $f$. 
\begin{enumerate}
\item(Bias bounds)  There exits a constant $C_b$ independent of $f,\varrho$ and $\tau$ such that 
\begin{subequations}
\begin{align}\label{eq:over_bias_image}
\|f-f_\alpha\|_{\XA} & \leq C_b \varrho^\frac{u}{(1-\theta)u+2\theta} \alpha^\frac{\theta}{(1-\theta)u+2\theta}, \\ 
\nonumber
\|f-f_\alpha\|_{\XL} & \leq C_b \varrho^\frac{(1-\xi)u+2\xi}{(1-\theta)u+2\theta}\alpha^\frac{\theta-\xi}{(1-\theta)u+2\theta} \quad\text{and}\\  \label{eq:over_r_bound}
\|f_\alpha \|_{\XR}&  \leq C_b \varrho^\frac{2}{(1-\theta)u+2\theta} \alpha^\frac{\theta-1}{(1-\theta)u+2\theta}
\end{align}
\end{subequations}
holds true for all $0<\alpha<\varrho^{-\frac{(1-\theta)u+2\xi}{\theta-\xi}} \tau^\frac{(1-\theta)u+2\theta}{\theta-\xi} $ and $f_\alpha\in S_\alpha(F(f))$ (see \eqref{eq:Tik_det}).
\item(Rates with a priori choice of $\alpha$) Let $0<c_l \leq c_r$. Suppose $\gobs\in \Y$
satisfies \eqref{eq:det_noise_level} with 
$0<\delta <\varrho^{-\frac{\xi}{\theta-\xi}} \tau^\frac{\theta}{\theta-\xi}$. 
Let $\alpha>0$ and $ \fh \in S_\alpha(\gobs)$.
There exists a constant $C_c$ independent of $f,\gobs, \varrho$, $\tau$ and $\delta$ such that 
\[  c_l \varrho^{-\frac{u}{\theta}} \delta^\frac{(1-\theta)u+2\theta}{\theta} \leq \alpha\leq c_r \varrho^{-\frac{u}{\theta}} \delta^\frac{(1-\theta)u+2\theta}{\theta}  \] 
implies the bounds
\begin{subequations}
\begin{align*}
\|f-\fh\|_{\XA} & \leq C_c \delta, \\ 
\|f-\fh\|_{\XL} & \leq C_c \varrho^\frac{\xi}{\theta} \delta^\frac{\theta-\xi}{\theta} \quad\text{and}\\ 
\|\fh \|_{\XR}&  \leq C_c \varrho^\frac{1}{\theta} \delta^\frac{\theta-1}{\theta}.
\end{align*}
\end{subequations}
\item (Rates with discrepancy principle) Let $1 < c_{\mathrm{D}} \leq C_{\mathrm{D}}$. Suppose $0<\delta <\varrho^{-\frac{\xi}{\theta-\xi}} \tau^\frac{\theta}{\theta-\xi}$, $\gobs\in \Y$ with $\|\gobs- F(f)\|_\Y\leq \delta$. Let $\alpha>0$ and $ \fh \in S_\alpha(\gobs)$. There exists a constant $C_d$ independent of $f,\gobs, \varrho$ and $\delta$ such that 
\[ c_{\mathrm{D}} \delta \leq \| \gobs - F(\fh)\|_\Y \leq C_{\mathrm{D}} \delta\] 
implies the following bounds  
\begin{align*}
\|f-\fh\|_{\XL} & \leq C_d \varrho^\frac{\xi}{\theta}\delta^\frac{\theta-\xi}{\theta} \quad\text{and} \\ 
\|\fh \|_{\XR}&  \leq C_d \varrho^\frac{1}{\theta} \delta^\frac{\theta-1}{\theta}.
\end{align*}
\end{enumerate}
\end{theorem}
\begin{proof}
Let $(f_t)_{t>0}$ be as in Lemma \ref{lem:aux_elements}. 
\begin{enumerate}
\item We choose 
\begin{align}\label{eq:choise_of_t}
 t=C_L^{-\frac{1}{\theta-\xi}} \varrho^\frac{u-2}{(1-\theta)u+2\theta} \alpha^\frac{1}{(1-\theta)u+2\theta}
\end{align} with $C_L$ from  Lemma \ref{lem:aux_elements}. Inequality \eqref{eq:aux_loss} yields 
\begin{align}\label{eq:over_xl_start}
 \|f-f_t\|_{\XL}\leq C_L \varrho t^{\theta-\xi} = \varrho^\frac{(1-\xi)u+2\xi}{(1-\theta)u+2\theta}\alpha^\frac{\theta-\xi}{(1-\theta)u+2\theta}<\tau.
\end{align} 
Hence $f_t\in D_F$, i.e.\ we may insert $f_t$ into the Tikhonov functional and use the Lipschitz condition of $F$, \eqref{eq:aux_image} and \eqref{eq:aux_penalty} to wind up with 
\begin{align*}
\frac{1}{2\alpha} \| F(f)-F(f_\alpha)\|_\Y^2 + \frac{1}{u} \|f_\alpha\|_\XR^u & \leq \frac{1}{2\alpha} \| F(f)-F(f_t)\|_\Y^2 + \frac{1}{u} \|f_t\|_\XR^u  \\ 
 & \leq  \frac{M_2^2}{2\alpha} \|f-f_t\|_\XA ^2 + \frac{1}{u} \|f_t\|_\XR^u \\ 
 & \leq    \frac{2 M_2^2}{\alpha} \varrho^2 t^{2\theta} + \frac{2^u}{u} \varrho^u t^{(\theta-1)u} \\ 
 &= c_1 \varrho^\frac{2u}{(1-\theta)u+2\theta} \alpha^\frac{(\theta-1)u}{(1-\theta)u+2\theta}
\end{align*}
with $c_1$ depending on $M_2, C_L, u, \theta$ and $\xi$.
We neglect the penalty term and use the Lipschitz condition of the inverse of $F$ to obtain the first bound
\[ \|f-f_\alpha\|_{\XA}\leq M_1  \| F(f)-F(f_\alpha)\|_\Y \leq  (2 c_1)^\frac{1}{2} M_1 \varrho^\frac{u}{(1-\theta)u+2\theta} \alpha^\frac{\theta}{(1-\theta)u+2\theta}.\]
Together with \eqref{eq:aux_image} we record 
\[ \| f_t - f_\alpha \|_\XA \leq \| f - f_t \|_\XA + \| f - f_\alpha \|_\XA \leq c_2 \varrho^\frac{u}{(1-\theta)u+2\theta} \alpha^\frac{\theta}{(1-\theta)u+2\theta} \] 
with $c_2$ depending on $C_L, c_1, M_1, \theta$ and $\xi$.\\
Neglecting the data fidelity term in the above estimation of the Tikhonov functional provides 
\[ \|f_\alpha\|_{\XR} \leq  (c_1 u)^\frac{1}{u} \varrho^\frac{2}{(1-\theta)u+2\theta} \alpha^\frac{\theta-1}{(1-\theta)u+2\theta}.\] 
Furthermore, we see that $\|f_t\|_{\XR}$ satisfies the same upper bound.  With the triangle inequality in $\XR$ we combine 
\begin{align*}\label{eq:over_bias_unbounded_XR}
 \| f_t - f_\alpha \|_\XR \leq c_{\XR} \left(\|f_t\|_{\XR}+ \|f_\alpha\|_{\XR} \right) \leq 2 c_{\XR} (c_1 u)^\frac{1}{u} \varrho^\frac{2}{(1-\theta)u+2\theta} \alpha^\frac{\theta-1}{(1-\theta)u+2\theta}.
\end{align*} 
Next, the interpolation inequality ${\|\cdot \|_{\XL}\leq c_3 \|\cdot \|_\XA^{1-\xi} \cdot \|\cdot \|_\XR^\xi}$ (see Lemma \ref{lem:intrerpolation_inequality}) furnishes 
\[ \|f_t-f_\alpha \|_\XL \leq c_4 \varrho^\frac{(1-\xi)u+2\xi}{(1-\theta)u+2\theta} \alpha^\frac{\theta-\xi}{(1-\theta)u+2\theta}  \]
with $c_4$ depending on $c_1$, $c_2$, $c_3$, $c_{\XR}$, $u$ and $\xi$.
Together with \eqref{eq:over_xl_start} we finally obtain 
\[ \|f-f_\alpha\|_{\XL} \leq \|f-f_t \|_\XL +\|f_t-f_\alpha\|_\XL \leq (1+c_4) \varrho^\frac{(1-\xi)u+2\xi}{(1-\theta)u+2\theta}\alpha^\frac{\theta-\xi}{(1-\theta)u+2\theta}. \] 
\item Taking  $t= C_L^{-\frac{1}{\theta-\xi}}\varrho^{-\frac{1}{\theta}} \delta^\frac{1}{\theta}$ we have
\begin{align}\label{eq:over_apri_start}
\|f-f_t\|_{\XL}\leq C_L \varrho t^{\theta-\xi}= \varrho^\frac{\xi}{\theta} \delta^\frac{\theta-\xi}{\theta}< \tau.
\end{align}
This ensures $f_t\in D_F$.
We insert into the Tikhonov functional, use the elementary inequality $(a+b)^2\leq 2 a^2 + 2 b^2$ for $a,b\geq 0$, \eqref{eq:aux_image}, \eqref{eq:aux_penalty}, the Lipschitz condition of $F$ and the choice of $\alpha$ to estimate  
\begin{align*}
\frac{1}{2\alpha} \| \gobs - F(\fh)\|_\Y^2  & + \frac{1}{u}\|\fh \|_\XR ^u  \leq 
\\& \leq\frac{1}{2\alpha} \|\gobs -F(f)+F(f)- F(f_t)\|_\Y^2  +\frac{1}{u}\|f_t\|_\XR^u \\ 
&\leq \frac{\delta^2}{\alpha}+ \frac{M_2^2}{\alpha} \|f-f_t \|_\XA ^2  + \frac{2^u}{u} \varrho^u t^{(\theta-1)u} \\  &\leq (1+4M_2^2C_L^{-\frac{2\theta}{\theta-\xi}}) \frac{\delta^2}{\alpha} + \frac{2^u}{u}C_L^\frac{(1-\theta)u}{\theta-\xi}\varrho^\frac{u}{\theta}\delta^\frac{(\theta-1) u}{\theta}\\ 
& \leq  c_5 \varrho^\frac{u}{\theta}\delta^\frac{(\theta-1) u}{\theta}
\end{align*}
with depending on $c_l, M_2, C_L, u, \theta$ and $\xi$.

Now we follow the argument in $(a)$: From the last inequality and the triangle inequality in $\Y$ we get 
\begin{align*}
\|f-\fh\|_\XA & \leq M_1 \|F(f)-\gobs+\gobs -F(\fh)\|_\Y  \\ & \leq M_1\delta
+M_1 (2c_5)^\frac{1}{2}  \alpha^\frac{1}{2}\varrho^\frac{u}{2\theta} \delta^\frac{(\theta-1) u}{2\theta} \\ 
& \leq M_1 (1+ (2c_5c_r)^\frac{1}{2}) \delta\
\end{align*}
which together with \eqref{eq:aux_image} implies
 ${\|f_t-\fh\|_\XA\leq c_6 \delta}$ with $c_6$ depending on $C_L,$ $M_1,$  $c_5,c_r$ $\theta$ and $\xi$.
Moreover,
  $\|\fh\|_\XR, \|f_t\|_\XR \leq (c_5 u)^\frac{1}{u} \varrho^\frac{1}{\theta} \delta^\frac{\theta-1}{\theta}.$
Hence  \[\|f_t-\fh\|_\XR\leq 2c_{\XR} (c_5 u)^\frac{1}{u} \varrho^\frac{1}{\theta} \delta^\frac{\theta-1}{\theta}\] by the triangle inequality in $\XR$. \\
We use the above interpolation inequality to combine the last two inequalities to $\|f_t-\fh\|_\XL \leq c_7 \varrho^\frac{\xi}{\theta} \delta^\frac{\theta-\xi}{\theta} $ with $c_7$ depending on $c_3$, $c_6$, $c_5$, $c_{\XR}$, $u$ and $\xi$.
With \eqref{eq:over_apri_start} we conclude
\[ \|f-\fh\|_\XL \leq (1+c_7) \varrho^\frac{\xi}{\theta} \delta^\frac{\theta-\xi}{\theta}.\]
\item We set $\varepsilon:=\min\left\{\frac{c_{\mathrm{D}}^2-1}{2},4 M_2 C_L^{-\frac{2\theta}{\theta-\xi}}\right\}$. Then $\varepsilon>0$. Furthermore, we take  
\[ t=\left(\frac{(4 M_2)\inv \varepsilon }{1+\varepsilon\inv}\right)^\frac{1}{2\theta}  \varrho^{-\frac{1}{\theta}}  \delta^\frac{1}{\theta}. \] 
Then \eqref{eq:aux_image} reads as 
\begin{align}\label{eq:over_discrep_start_ima}
 \|f-f_t\|_\XA \leq 2\varrho t^\theta= \left(\frac{\varepsilon}{1+\varepsilon\inv}\right)^\frac{1}{2}M_2^{-\frac{1}{2}} \delta.
\end{align}
Due to \eqref{eq:aux_loss} we obtain
\begin{align}\label{eq:over_discrep_start_loss}
\|f-f_t\|_{\XL}\leq C_L \varrho t^{\theta-\xi}\leq  C_L \left((4 M_2)\inv \varepsilon\right)^\frac{\theta-\xi}{2\theta} \varrho^\frac{\xi}{\theta} \delta^\frac{\theta-\xi}{\theta}   \leq  \varrho^\frac{\xi}{\theta} \delta^\frac{\theta-\xi}{\theta}< \tau
\end{align}
which provides $f_t\in D_F$. \\
 In the following we use  the elementary inequality 
$ (a+b)^2 \leq (1+\varepsilon)a^2 +(1 +\varepsilon\inv)b^2$ for all $a,b\geq 0 $ (which is proven by expanding the square and 
applying Young's inequality on the mixed term) and \eqref{eq:over_discrep_start_ima} to estimate 
\begin{align*}
\|\gobs -F(f_t)\|_\Y^2 & \leq (1+\varepsilon)\delta^2 + (1+\varepsilon\inv)\|F(f)-F(f_t)\|_\Y^2  \\
&\leq (1+\varepsilon)\delta^2 + (1+\varepsilon\inv)M_2\|f-f_t\|_\XA^2 \\
& \leq  (1+2\varepsilon)\delta^2 \leq c_{\mathrm{D}}^2 \delta^2 \leq \|\gobs -F(\fh)\|_\Y^2.
\end{align*}
Therefore, a comparison of the Tikhonov functional taken at $\fh$ and $f_t$, and \eqref{eq:aux_penalty} yield 
\[ \|\fh\|_\XR\leq  \|f_t \|_\XR \leq 2\varrho t^{\theta-1}= c_8 \varrho^\frac{1}{\theta} \delta^\frac{\theta-1}{\theta} \] $c_8$ depending on $M_2,$ $\varepsilon$ and $\theta$. 
Hence $\|f_t-\fh\|_\XR\leq 2 c_{\XR} c_8 \varrho^\frac{1}{\theta} \delta^\frac{\theta-1}{\theta}$. Moreover, 
\[  \|g_t-F(\fh)\|_\Y \leq \|\gobs -g_t\|_\Y + \|\gobs -F(\fh)\|_\Y \leq 2  C_{\mathrm{D}} \delta.\]
Therefore, $\|f_t-\fh\|_\XA \leq  M_1 C_{\mathrm{D}} \delta$ by the Lipschitz condition. 
As above we conclude $\|f_t-\fh\|_\XL \leq c_9 \varrho^\frac{\xi}{\theta} \delta^\frac{\theta-\xi}{\theta} $ with $c_9$ depending on $c_3$, $C_{\mathrm{D}}$, $c_8$, $c_{\XR}$ and $\xi$ and use \eqref{eq:over_discrep_start_loss} to finish up with 
\[ \|f-\fh\|_\XL \leq (1+c_9)\varrho^\frac{\xi}{\theta} \delta^\frac{\theta-\xi}{\theta}.   \]  
\end{enumerate}
\end{proof}

We discuss our result in a series of remarks. 
\begin{remark}[Interior point]\label{rem:open ball} 
The requirement that $f$ be an interior point of the domain in $\XL$ may be weakened to the requirement that elements $f_t$ satisfying the bounds given in Lemma \ref{lem:aux_elements} belong to $D_F$ for $t$ small enough.
\end{remark} 
\begin{remark}[Influence of the exponent $u$] 
A strength of the above theorem is that it provides convergence rates for all exponents $u$. Note that the choice of $u$ does not influence the rate while it does influence the bias bounds and the parameter choice rule. An inspection of the a priori rule shows that a larger $u$ allows for a larger choice of the parameter $\alpha$. 
The flexibility in the choice of $u$ in our theory is a remarkable difference to many other variational convergence theories where one has to pick a specific exponent (see, e.g., \cite{HM:18, WSH:20}). 
The authors also do not expect any difficulty in generalizing this result to other exponents than $2$ in the data fidelity. 
\end{remark} 

\begin{remark}[Equivalent norms] 
The presented theory relies on a purely quasi-Banach space theoretic framework: As we do not appeal to any metric or convex notions like subdifferentials or convexity the result in Theorem \ref{thm:rate_abstract} stays the same up to a change of the constants if we change the norm on any of the occurring spaces up to equivalence. This has an important impact on regularization with wavelet  penalties that we will discuss in the next section.  \\ 
Once again this is a major difference to classical variational regularization theory. For example, it is not clear how the subdifferential of a norm involved in the source condition $A^\ast \omega\in \partial \mathcal{R} (f)$ for a linear operator $A$ changes if the norm is replaced by an equivalent one. Also classical variational source conditions are 
characterized by the smoothness of $\partial \mathcal{R}$ rather than the smoothness 
of $f$ (see \cite{WSH:20}), and the former may change if the norm in the penalty term is replaced by an equivalent norm. 
\end{remark}
\begin{remark}[Converse result]\label{rem:over_converse}
Suppose minimizers in \eqref{eq:Tik_det} exist for all $g\in \Y$ and $\alpha>0$ and $D_F=\XA$. In view of Remark \ref{rem:aux_converse} one can reclaim $f\in (\X_-,\XR)_{\theta,\infty}$ from the bias bound \eqref{eq:over_bias_image} together with \eqref{eq:over_r_bound} as with $\alpha(t)=c \varrho^{2-u}t^{(1-\theta)u+2\theta}$ for a suitable choosen constant $c$ depending only on $C_b$ a net $(f_{\alpha(t)})_{t>0}$ with $f_{\alpha(t)}\in S_{\alpha(t)}$ satisfies the bounds \eqref{eq:aux_image} and \eqref{eq:aux_penalty} in Lemma \ref{lem:aux_elements}.
\end{remark}

\begin{remark}[Limiting case $\theta=1$]
In the case $\theta=1$ the parameter choice rule in Theorem \ref{thm:rate_abstract} becomes $\alpha\sim \delta^2$. Here the results provides boundedness of the estimators $f_\alpha$ and $\fh$ in $\XR$. Due to  Proposition \ref{app:interpolation_limiting} we have $\XR\subset  \left(\X_-,\XR\right)_{1,\infty}$. The latter two spaces agree if $\XR$ is reflexive (see \cite[1.3.2. Rem.~2]{Triebel1978}). 
\end{remark} 

Before we illustrate our theorem by simple sequence space models, let us point out 
that in contrast to \cite{HM:18,CHY:21} we do not need to require that 
$c_{\mathrm{D}}=C_{\mathrm{D}}$ in the discrepancy principle. As also mentioned in \cite{CHY:21}, 
this is desirable in view of practical implementations.    

\begin{example}[Embedding operators in sequence spaces]
\begin{itemize}
\item 
Let $p \in (0,2)$ and $u\in (0,\infty)$. We consider $\XR=\ell^p$, $\Y=\XL=\XA=\ell^2$, \linebreak $F\colon \XR\rightarrow \Y$ with $x\mapsto x$ the embedding operator. Then Assumption \ref{ass:oversmooth} holds true with $\xi=0$. \\
Let $v\in (p,2)$, then we obtain 
\[ \left(\ell^2, \ell^p \right)_{\theta_v,\infty}= \omega\ell^v \quad\text{with} \quad\theta_v=\frac{p(2-v)}{v(2-p)} \]
(see e.g. \cite{F:78}).
Here $\omega\ell^v$ stands for the weak $\ell^v$-space given by the quasi-norm 
\[ \|x\|_{\omega\ell^v}^v = \sup_{\alpha>0} \alpha^v \# \left\{ |x_k| >\alpha\right\}. \]
Theorem \ref{thm:rate_abstract} yields that $x\in \omega\ell^v$ implies 
\[ \|x-\xh \|_2 = \mathcal{O}(\delta) \quad\text{and}\quad \|\xh\|_p= \mathcal{O}\left( \delta^\frac{2(p-v)}{p(2-v)}\right) \]
with $\xh\in \argmin_{z \in \ell^p } \infbracket{ \frac{1}{2\alpha} \|x^\mathrm{obs} - z\|_2^2 + \frac{1}{u}\|z\|_p^u } $, and $\|x-x^\mathrm{obs}\|_2\leq \delta$, and either of the parameter choice rules specified in Theorem \ref{thm:rate_abstract}.	
\item Once again let $p \in (0,2)$ and $u\in (0,\infty)$. Now we consider $\XR=\ell^p$, $\Y=\XA=\ell^\infty$, $\XL=\ell^2$ and again $F\colon \XR\rightarrow \Y$ the embedding operator. With $\xi=\frac{p}{2}$ the continuous embedding \[\left(\ell^{\infty}, \ell^p\right)_{\xi,1}\subset  \left(\ell^{\infty}, \ell^p\right)_{\xi,2}=\ell^2  \]  yields Assumption \ref{ass:oversmooth}. \\ 
For $v\in (p,2)$ we have
 $ \left(\ell^\infty, \ell^p \right)_{p/v,\infty}=\omega\ell^v$. Hence for $x\in \omega\ell^v$ we obtain 
\[  \|x-\xh \|_2 = \mathcal{O}(\delta^\frac{2-v}{2}) \quad\text{and}\quad \|\xh\|_p= \mathcal{O}\left( \delta^\frac{p-v}{p}\right) \]
with $\xh\in \argmin_{z \in \ell^p } \infbracket{ \frac{1}{2\alpha} \|x^\mathrm{obs} - z\|_{\infty}^2 + \frac{1}{u}\|z\|_p^u } $, and $\|x-x^\mathrm{obs}\|_\infty\leq \delta$, and either of the parameter choice rules specified in Theorem \ref{thm:rate_abstract}.	
\end{itemize}
\end{example}

\section{Besov space regularization}\label{sec:besov} 
In this section we apply Theorem \ref{thm:rate_abstract} to regularization of finitely smoothing operators with Besov space penalty term. 
For a comprehensive treatment of Besov spaces we refer to \cite{Triebel1992,triebel:08, 
Triebel2010} and also to \cite[Ch.~4]{GN:15} for a self-contained introduction and applications in statistics. Besov space $B^s_{p,q}(\mathbb{R}^d)$ for a smoothness 
index $s\in\mathbb{R}$, an integrability index $p\in(0,\infty]$ 
and a fine index $q\in (0,\infty]$ with quasi-norms $\|\cdot\|_{B^s_{p,q}(\mathbb{R}^d)}$ 
can be defined in several equivalent ways, among others via a 
dyadic partition of unity in Fourier space, via the modulus of continuity or 
via wavelet decompositions. In contrast to the analysis of non-oversmoothing Besov regularization in \cite{HM:19, Miller2021,MH:20, WSH:20}, it will not matter here, 
which of these equivalent norms is used in the following. 

In the following let $\Omega\subset \mathbb{R}^d$ be 
a bounded Lipschitz domain. Then $\Bspace rpq:= \{f|_{\Omega}:f\in B^s_{p,q}(\mathbb{R}^d)\}$ with $\Bn{g} rpq := \inf \{\|f\|_{B^s_{p,q}(\mathbb{R}^d)}: f|_{\Omega}=g\}$ 
is a quasi-Banach space, and even a Banach space if $p,q\geq 1$ (see \cite{triebel:08}). 
Some properties of these spaces and relations to other function spaces 
are summarized in \Cref{sec:app_besov}.

Throughout this section we use $\XR:=\Bspace rpq$ for fixed $r>0$ and $p,q\in (0,\infty]$  and  consider the regularization scheme 
\begin{align}\label{eq:Tik_det_besov}
S_\alpha (g) = \argmin_{h \in \tilde{D}_F } \infbracket{ \frac{1}{2 \alpha} \|g- F(h)  \|_{\Y}^2  + \frac{1}{u} \Bn{h} rpq^u}, \qquad g\in \Y
\end{align}
for a fixed exponent $u\in (0,\infty)$. A natural choice is $u=q$. 

\subsection{Convergence rate result}
We first formulate our assumptions on the forward operator.
Recall that $\Bspace s22 = W^s_2(\Omega)$ with equivalent norms for all $s\in \mathbb{R}$ (Proposition \ref{app:relation_lp_sob}).
\begin{assumption}\label{ass:besov} 
Suppose that $a\geq 0$ and $\Bspace rpq\subset \Bspace {-a}22$ with continuous embedding. 
Let $\tilde{D}_F\subset \Bspace rpq$, $\Y$ be a Banach space and $F\colon \tilde{D}_F \rightarrow \Y$ be a map satisfying 
\begin{align*}
\frac{1}{M_1} \Bn {f_1-f_2}{-a}22 \leq \|F(f_1)- F(f_2) \|_\Y  \leq M_2 \Bn {f_1-f_2}{-a}22\quad\text{ for all } f_1,f_2 \in \tilde{D}_F.
\end{align*}
for constants $M_1, M_2 >0.$
\end{assumption} 

The assumption of a continuous embedding $\Bspace rpq\subset \Bspace {-a}22$ is satisfied 
if $a+r>d\big(\frac{1}{p}-\frac{1}{2}\big)$ (see \ref{eq:besov_embed_q_neq})
. For $q=2$ even the condition 
$a+r\geq d\big(\frac{1}{p}-\frac{1}{2}\big)$ suffices (see \ref{eq:besov_embed_q_eq}). 

Now we state and prove the convergence rate result for oversmoothing Besov space regularization. We first state our theorem under the abstract smoothness condition given by the maximal real interpolation space in Theorem \ref{thm:rate_abstract} and discuss how to find more handy smoothness conditions in terms of Besov spaces afterwards. For the sake of brevity we do not state the bounds 
on the bias.
\begin{corollary}[Rates for oversmoothing Besov space regularization]\label{cor:over_besov}
Consider the regularization scheme \eqref{eq:Tik_det_besov} for some 
$p,q\in (0,\infty]$ with $q\leq p$, $r>0$ and $u\in (0,\infty),$ 
such that $\overline{p}:=\frac{2p(a+r)}{2a+pr}\geq 1$ (i.e.\ $p\geq \frac{2a}{2a+r}$)
and suppose Assumption \ref{ass:besov} holds true. 
Assume the true solution $f$ has smoothness index  $s\in (0,r]$ in the sense that 
\begin{align}\label{eq:besov_regularity}
f\in \left(\Bspace {-a}22,\Bspace rpq\right)_{\theta_s,\infty}\quad\text{for}\quad \theta_s:=\frac{s+a}{a+r} \quad \mbox{and}\quad
\|f\|_{\left(\Bspace {-a}22,\Bspace rpq\right)_{\theta_s,\infty}}\leq\varrho.
\end{align}  
for some $\varrho>0$ (see also Remark \ref{rem:over_besov_smoothness}). 
Suppose that the closure $D_F$ of $\tilde{D}_F$ in $\Bspace {-a}22$ contains a $\Bspace 0{\overline{p}}{\overline{p}}$-ball with radius $\tau>0$ around $f$. 
Suppose $\gobs\in \Y$ satisfies \eqref{eq:det_noise_level} for $0<\delta <\varrho^{-\frac{a}{s}} \tau^\frac{s+a}{s}$ and $ \fh \in S_\alpha(\gobs)$ defined in \eqref{eq:Tik_det_besov} for some $\alpha>0$. Let $0<c_l \leq c_r$ and $1 < c_{\mathrm{D}} \leq C_{\mathrm{D}}$. 
Then there is a constant $C_r$ independent of $f,\gobs, \varrho$, $\tau$ and $\delta$ such that either of the conditions
\[  c_l \varrho^{-\frac{u(a+r)}{s+a}} \delta^\frac{(2-u)s+2a+ur}{s+a} \leq \alpha\leq c_r \varrho^{-\frac{u(a+r)}{s+a}} \delta^\frac{(2-u)s+2a+ur}{s+a} \text{ and } \] \[ c_{\mathrm{D}} \delta \leq \| \gobs - F(\fh)\|_\Y \leq C_{\mathrm{D}} \delta \] 
on the choice of $\alpha$ implies the following bounds:
\begin{subequations}
\begin{align}
\Bn {f-\fh} {-a}22  & \leq C_r \delta, \\ 
\label{eq:main_besov_rate}
\Bn {f-\fh} {0}{\overline{p}}{\overline{p}} & \leq C_r \varrho^\frac{a}{s+a} \delta^\frac{s}{s+a}\\ 
\Bn {\fh} rpq&  \leq C_r \varrho^\frac{a+r}{s+a} \delta^\frac{s-r}{s+a}
\end{align}
\end{subequations}
\end{corollary}
\begin{proof}
We set  $\xi:=\frac{a}{a+r}$, $\XR=\Bspace rpq$, $\X_-:=\Bspace {-a}22$ and $\XL:=\Bspace 0 {\overline{p}}{\overline{p}}$, and verify Assumption \ref{ass:oversmooth}. The two-sided Lipschitz condition holds true due to Assumption \ref{ass:besov}.
 If $a=0$, then $\xi=0$ and we have $\overline{p}=2$. Therefore, $\X_-=\XL=\Bspace 022=L^2(\Omega)$. 
If $a>0$ we use $q\leq p$ and $1\leq \overline{p}$ to obtain the following chain of continuous embeddings:
\begin{align}\label{eq:over_besov_loss_inter}
\left(\Bspace {-a}22 , \Bspace rpq \right)_{\xi,1} 
\subset \left(\Bspace {-a}22 , \Bspace rpq \right)_{\xi,\overline{p}}\subset \left(\Bspace {-a}22 , \Bspace rpp \right)_{\xi,\overline{p}}=\Bspace 0{\overline{p}}{\overline{p}},
\end{align}
see \cite[Thm.~3.4.1.(b)]{Bergh1976} for the first embedding, 
\eqref{eq:besov_embed_sp_fixed} 
for the second, and \eqref{eq:besov_interp_p_eq_q} for the interpolation identity.  
This shows Assumption \ref{ass:oversmooth}, i.e.\ $\left(\XA , \XR \right)_{\xi,1} \subset \XL$,   and the result follows from Theorem \ref{thm:rate_abstract}. 
\end{proof}\\
In contrast to the analysis in \cite{HM:19, Miller2021,MH:20, WSH:20}, which is restricted to certain choices of $r$,$p$,$q$ and $u$, our only restrictions on the parameters are $r>0$, $q\leq p$, and $\overline{p}\geq 1$. 
We will see that the assumption $q\leq p$ can be dropped by some refined argument using a complex interpolation identity. 

We further discuss our result in the following remarks.
\begin{remark}[$L^{\overline{p}}$-loss]\label{rem:over_Lp_loss}
Suppose $p\leq 2$. Then $\overline{p}\leq 2$. Hence the continuous embedding $\Bspace 0 {\overline{p}}{\overline{p}}\subset L^{\overline{p}}(\Omega)$ (see \eqref{eq:besov_lp_pgeq1}) 
together with \eqref{eq:over_besov_loss_inter} yields $\left(\Bspace {-a}22 , \Bspace rpq \right)_{\xi,1} \subset L^{\overline{p}}(\Omega)$. Therefore, Corollary \ref{cor:over_besov} remains valid word for word if one replaces $\Bspace 0 {\overline{p}}{\overline{p}}$ by $L^{\overline{p}}(\Omega)$ in this case.
\end{remark}

\begin{remark}[Smoothness condition]\label{rem:over_besov_smoothness}
Suppose that $p,q\geq 1$. 
By the complex interpolation \eqref{eq:complex_besov_interp} we have: 
\begin{align}\label{eq:over_besov_complex_inter}
\left[\Bspace {-a}22, \Bspace rpq \right]_{\theta_s}=\Bspace s{p_s}{q_s} ,\,  p_s=\frac{2p(a+r)}{s(2-p)+2a+pr},\, q_s=\frac{2q(a+r)}{s(2-q)+2a+qr}.
\end{align}
With this we obtain a continuous embedding 
${\Bspace s{p_s}{q_s}\subset \left(\Bspace {-a}22, \Bspace rpq \right)_{\theta_s,\infty}}$
as for Banach spaces 
the complex interpolation space $\left[ \cdot , \cdot \right]_\theta$ is always continuously embedded in the real interpolation space $\left( \cdot , \cdot \right)_{\theta,\infty}$  (see \cite[Thm.~4.7.1.]{Bergh1976}).
Hence the statements in Corollary \ref{cor:over_besov} remain true if the smoothness assumption on $f$ formulated in terms of $\left(\Bspace {-a}22, \Bspace rpq \right)_{\theta_s,\infty}$ is replaced by 
\begin{align}\label{eq:over_besov_smoothness}
f\in \Bspace s{p_s}{q_s}\qquad \mbox{with}\qquad 
\Bn f s{p_s}{q_s}\leq\varrho.
\end{align} 
\end{remark}

\begin{remark}[Assumption $q\leq p$]
We also comment on the assumption $q\leq p$, again for the Banach space case 
$p,q\geq 1$: Using complex interpolation this restriction can be dropped as follows.
 Since the real interpolation space $\left( \cdot , \cdot \right)_{\theta,1}$ is always continuously embedded in  complex interpolation space (see \cite[Thm.~4.7.1.]{Bergh1976}) identity \eqref{eq:over_besov_complex_inter} yields a continuous embedding 
\[ \left(\Bspace {-a}22, \Bspace rpq \right)_{\xi,1} \subset\Bspace 0 {\overline{p}}{\overline{q}}\]  for $\overline{p}$ as in Corollary \ref{cor:over_besov} and $\overline{q}:=\frac{2q(a+r)}{2a+qr}$. Hence the statements in Corollary \ref{cor:over_besov} remain true in the case $q>p$ if one replaces $\Bspace 0{\overline{p}}{\overline{p}}$ by $\Bspace 0{\overline{p}}{\overline{q}}$.
\end{remark}


\begin{remark}[Other domains and boundary conditions]
For the sake of clarity we have confined ourselves to bounded Lipschitz domains 
$\Omega\subset\mathbb{R}^d$ and to the Besov spaces $\Bspace spq$. 
However, Corollary \ref{cor:over_besov} only relies on the interpolation identity 
\eqref{eq:besov_interp_p_eq_q}, the embedding \eqref{eq:besov_embed_sp_fixed}, 
and the embedding stated in Assumption \ref{ass:besov}. These are also valid 
in many other situations (sometimes under additional assumptions), 
e.g.\ for certain unbounded domains (in particular $\mathbb{R}^d$ and half-spaces, 
see \cite{Triebel2010}), certain Riemannian manifolds 
(see \cite[Chapter 7]{Triebel1992}) as well as Besov spaces with other 
boundary conditions (see \cite[Chapter 4]{Triebel1978}). 
\end{remark}
\begin{example}[Hilbert spaces] 
For $p=q=u=2$ the regularization scheme \eqref{eq:Tik_det_besov} becomes classical Tikhonov regularization with $W^r_2(\Omega)=\Bspace r22$ penalty. Here we obtain 
\[ \|\fh-f\|_{L^2} = \mathcal{O}(\rho^{\frac{a}{s+a}}\delta^{\frac{s}{s+a}})
\qquad \mbox{if}\qquad \|f\|_{B^s_{2,\infty}} \leq \rho, \quad s\in (0,r).\]
Due to $W^s_2(\Omega)=B^s_{2,2}(\Omega)\subset B^s_{2,\infty}(\Omega)$ this reproduces the results in \cite{natterer:84} and \cite{HM:18}.
\end{example} 

\subsection{Sparsity promoting wavelet regularization}\label{sec:Besov_wavelet}
In the following we explain how regularization by wavelet penalization and in particular weighted $\ell^1$-regularization of wavelet coefficients is contained in our setup. 
The latter is often used since it leads to sparse estimators in the sense that 
only a finite (and often small) number of wavelet coefficients of $\fh$ do not vanish. 

We introduce the scale of Besov sequence spaces $\bspace spq$ that allows to characterize Besov function spaces $\Bspace spq$ by decay properties of coefficients in wavelet expansions (see also \cite[Def.~2.6]{triebel:08}). Let $c_\Lambda, C_\lambda>0$ and $(\Lambda_j)_{j\in\nat}$ be a family of finite sets such that 
\begin{align*}
c_\Lambda 2^{jd}\leq |\Lambda_j|\leq C_\Lambda 2^{jd} \quad \text{ for all } j\in\nat.
\end{align*}
We consider the index set \[ \Lambda:= \{(j,k) \colon j\in\mathbb{N}_0, k\in\Lambda_j\}.\] For a sequence $x=\left(x_{j,k}\right)_{(j,k)\in \Lambda}$ and a fixed $j\in \nat$ we denote by $x_j:= (x_{j,k})_{k\in \Lambda_j}\in \mathbb{R}^{\Lambda_j}$ the projection onto the $j$-th level.
For $s\in \mathbb{R}$ and $p,q \in [1,\infty]$ let us introduce
\[
\bspace spq:=\left\{x\in\mathbb{R}^\Lambda \colon \bn xspq<\infty\right\}\quad \text{with}\quad
\bn xspq:=\left\|
\left(2^{js}2^{jd\left(\frac{1}{2}-\frac{1}{p}\right)} 
\|x_j\|_p\right)_{j\in\mathbb{N}_0}\right\|_{q}. \]
Suppose $(\psi_\lambda)_{\lambda\in \Lambda}$ is a wavelet system on $\Omega$ such that the wavelet synthesis operator
\[ \wav \colon\bspace rpq \rightarrow \Bspace rpq \quad\text{ given by } (\wav x)(r) =\sum_{\lambda\in \Lambda} x_\lambda \psi_\lambda(r) \text{ for } r\in \Omega  \] 
is a norm isomorphism for $r,p,q$ the parameters involved in $\XR=\Bspace rpq$. 
In this case we use the norm $\Bn {f} rpq:= \bn {\wav\inv f} rpq$ in \eqref{eq:Tik_det_besov}. 
By transformation rules of $\argmin$ under composition with a bijective mapping,  
the estimators in \eqref{eq:Tik_det_besov} can then be rewritten in the form
\begin{align*}\label{eq:Tik_onehomo_wavelet}
S_\alpha(g) = \wav \argmin_{x\in \wav\inv (\tilde{D}_F)} \infbracket{ \frac{1}{2\alpha } \|g - F(\wav x)\|_{\Y}^2 +  \frac{1}{u} \bn xrpq^u }.
\end{align*}
This is the more common implementation of wavelet penalization methods. 
If there exists a wavelet analysis operator 
\[ \mathcal{A}\colon \Bspace rpq \rightarrow \bspace rpq \quad\text{satisfying }   \bn {\mathcal{A} \cdot }rpq \sim \Bn \cdot rpq.\] 
then $\|f\|_\XR := \bn {\mathcal{A}\cdot  } rpq$ is equivalent to $\Bn {\cdot }rpq$, 
and may be used as penalty term in the framework of Corollary \ref{cor:over_besov}. 

\begin{example}[$p=2$] \label{rem:over_2}
In the case $p=2$ we have $\left(\Bspace {-a}22, \Bspace r2q \right)_{\theta_s,\infty}=\Bspace s2\infty$ (see \eqref{eq:besov_interp_p_fixed}), and  Corollary \ref{cor:over_besov} shows that 
\begin{align}\label{eq:p2rate}
\|\fh-f\|_{L^2} = \mathcal{O}(\rho^{\frac{a}{s+a}}\delta^{\frac{s}{s+a}})
\qquad \mbox{if}\qquad \|f\|_{B^s_{2,\infty}} \leq \rho, \quad s\in (0,r).
\end{align}
The same convergence rate has been obtained for non-oversmoothing Besov 
wavelet penalization in \cite{WSH:20} for $q=u\geq 2$ and $s\in (0, \frac{a}{q-1}]$  
and in \cite{HM:19} for $q=u=1$ and $s\in (0,\infty)$ (for infinitely smooth wavelets). 
In \cite{WSH:20} it was shown that this rate is of optimal order.\\
As a reference example we discuss rates for piecewise smooth univariate functions with jumps. 
As shown in \cite[Ex.~30]{MH:20} such functions belong to $B^s_{p,\infty}$ if and only if 
$s\leq \frac{1}{p}$ and to $B^s_{p,q}$ with $q<\infty$ if and only if 
$s<\frac{1}{p}$. Hence, in our setting we have $s=\frac{1}{2}$ in \eqref{eq:p2rate}. 
\end{example}

\begin{example}[$p=q=1$]\label{ex:p1}
Note that for $u=p=q=1$ we obtain a weighted $\ell^1$-penalty. The largest 
smoothness class  
$ \left(\Bospace {-a}22(\Omega), \Bospace r11(\Omega) \right)_{\theta_s,\infty}=\mathcal{S}(k_s)$ 
was characterized in \cite{MH:20} as image of a weighted Lorentz sequence space $k_s$, 
 and a converse result was derived for this class. As this is not a Besov space, 
 we will work with the slightly smaller space 
$ \left(\Bospace {-a}22(\Omega), \Bospace r11(\Omega) \right)_{\theta_s,p_s}
=\Bospace {s}{p_s}{p_s}(\Omega) = W^s_{p_s}(\Omega)$ with 
$p_s = \frac{2a+2r}{2a+r+s}\in (1,\frac{2a+2r}{2a+r})$ for simplicity (see \eqref{eq:besov_interp_p_eq_q}, Prop.~\ref{app:relation_lp_sob}). 
Hence, Corollary \ref{cor:over_besov} implies that 
\begin{align}
\|\fh-f\|_{L^{\overline{p}}} = \mathcal{O}(\varrho^{\frac{a}{s+a}}\delta^{\frac{s}{s+a}})
\qquad \mbox{if}\qquad \|f\|_{W^s_{p_s}} \sim  
\|f\|_{B^s_{p_s,p_s}} \leq \rho, \quad s\in (0,r)
\end{align}
for $\overline{p} = \frac{2a+2r}{2a+r}$. This 
reproves results that were derived in \cite{MH:20} using hard-thresholding approximations 
of the true solution. \\
For piecewise smooth functions with jumps the condition $s<\frac{1}{p_s}$ is 
equivalent to $s<\frac{2a+r}{2a+2r-1}$, and the right hand side is always larger than 
$\frac{1}{2}$. Therefore, we obtain a faster rate for $p=1$ than for $p=2$ although 
only in the $L^{\overline{p}}$- rather than the $L^2$-norm. 
\end{example}

\begin{example}[$p<1$]\label{ex:pl1} 
For $p=q=u<1$ we obtain a weighted $\ell^p$-penalty. In analogy to 
Example~\ref{ex:p1} we use the smoothness class
$ \left(\Bospace {-a}22(\Omega), \Bospace rpp(\Omega) \right)_{\theta_s,\tilde{p}_s}
=\Bospace {s}{\tilde{p}_s}{\tilde{p}_s}(\Omega) = W^s_{\tilde{p}_s}(\Omega)$ with 
$\tilde{p}_s = \frac{2p(a+r)}{2a+pr+(2-p)s}\in (p,\frac{2p(a+r)}{2a+pr})$ and 
find that 
\begin{align}
\|\fh-f\|_{L^{\overline{p}}} = \mathcal{O}(\rho^{\frac{a}{s+a}}\delta^{\frac{s}{s+a}})
\qquad \mbox{if}\qquad \|f\|_{W^s_{\tilde{p}_s}} \sim  
\|f\|_{B^s_{\tilde{p}_s,\tilde{p}_s}} \leq \rho, \quad s\in (0,r)
\end{align}
for $\overline{p} = \frac{2p(a+r)}{2a+pr}$.\\
For piecewise smooth functions with jumps the condition $s<\frac{1}{\tilde{p}_s}$ is 
equivalent to $s<\frac{2a+r}{p(2a+2r+1)-2}$ (where the denominator is positive 
due to the first part of Assumption \ref{ass:besov}). Hence choosing $p<1$ 
rather than $p=1$ pays off 
in the sense that we obtain an even higher rate of convergence, but also in an 
even weaker norm.  
\end{example}

\begin{remark}['Does oversmoothing harm?']
\label{rem:over_hurt}
To conclude this section we point out a difference in the previous three examples. 
For $p=2$ our convergence rate analysis yields the same convergence rate $\mathcal{O}(\delta^\frac{s}{s+a})$ measured in the same norm, the $L^2$-norm, under the same smoothness condition given by $\Bspace s2\infty$ as in the case $r=0$. 
Hence the paradigm 'oversmoothing does not harm' known for Hilbert-space regularization remains true for $\Bspace r2q$ Banach space penalties with $q<2$. \\
In contrast, in Examples \ref{ex:p1} and \ref{ex:pl1}, a higher value of $r$ may cause an assignment of a lower smoothness $s$ to a fixed true solution. 
On the other hand, the error is than measured in a stronger norm. 
This indicates that the integrability index in the loss function norm 
may have an influence on the convergence rate. 
It calls for the development of a convergence rate theory that is more flexible in the choice of the loss function and allows for norms which cannot be sharply bounded 
by powers of the norms of the spaces $\X_-$ and $\XR$ in 
Assumption~\ref{ass:oversmooth} via interpolations. 
\end{remark}

\section{Bounded variation regularization}\label{sec:bv}
This section contains an application of Theorem \ref{thm:rate_abstract} to Tikhonov regularization with penalty term given by the $\mathrm{BV}$-norm. Let $d\in \mathbb{N}$ and $\Omega\subset \mathbb{R}^d$ a bounded Lipschitz domain. A function $f\in L^1(\Omega)$ has \emph{bounded variation} if 
\begin{align*}
|f|_{\mathrm{BV}(\Omega)}:=\sup\left\{\int_{\Omega}	f(x)\, \mathrm{div} g(x) \,\mathrm{d}x  \colon g\in C^1_c(\Omega,\mathbb{R}^d), \, \|g\|_{L^\infty(\Omega,\mathbb{R}^d)}\leq 1 \right\} <\infty.
\end{align*}
Here $\|g\|_{L^\infty(\Omega,\mathbb{R}^d)}:=\left\|\left(\sum_{j=1}^d g_i^2\right)^{1/2} \right\|_{L^\infty(\Omega)}$ with $g=(g_1,\ldots g_n)$. Then \[\BV(\Omega):=\{ f\in L^1(\Omega)\colon |f|_{\BV (\Omega)}<\infty\}\] is a Banach space equipped with $\|\cdot \|_{\mathrm{BV}(\Omega)}:= \|\cdot \|_{L^1(\Omega)}+ | \cdot|_{\BV(\Omega)}$.
We refer to \cite{ambrosio2000functions} for a detailed study of spaces of bounded variation.\\ 
For $a\geq 0$ with $a\geq \frac{d}{2}-1$ there is a continuous embedding $\BV (\Omega)\subset \Bspace {-a}22$ (see Proposition \ref{app:bv_embed}).   
In this section we will use the following assumption on the forward operator. 
\begin{assumption}\label{ass:bv} 
Let $a\geq 0$ with $a\geq \frac{d}{2}-1$. Let  
 $\tilde{D}_F\subset \BV (\Omega)$, $\Y$ be a Banach space and $F\colon \tilde{D}_F \rightarrow \Y$ be a map satisfying 
\begin{align*}
\frac{1}{M_1} \Bn {f_1-f_2}{-a}22 \leq \|F(f_1)- F(f_2) \|_\Y  \leq M_2 \Bn {f_1-f_2}{-a}22\quad\text{ for all } f_1,f_2 \in \tilde{D}_F.
\end{align*}
for constants $M_1, M_2 >0.$
\end{assumption} \\
For $g\in \Y$ we consider 
\begin{align}\label{eq:Tik_BV}
S_\alpha (g) = \argmin_{h \in \tilde{D}_F} \, \infbracket{ \frac{1}{2 \alpha} \|g- F(h)  \|_{\Y}^2  + \| h\|_{\mathrm{BV}} }.
\end{align}
We refer to \cite{acar1994analysis} for this kind of regularization scheme for linear operators including a proof of existence of minimizers and to \cite{del2019multiscale} for a treatment of similar estimators in a statistical setting.  \\
Let $a\geq 0$ and $s\in (-a,1)$.
The following interpolation identity, based on the result by Cohen et al. in \cite[Thm.~ 1.4]{cohen2003harmonic}, is a crucial ingredient for our convergence rates result 
\begin{align}\label{eq:BV_interpolation}
\Bspace s{t_s}{t_s} = \left( \Bspace {-a}22, \BV(\Omega) \right)_{\theta_s,t_s} \quad\text{ with }  \theta_s:= \frac{s+a}{a+1} \text { and }  t_s:= \frac{2a+2}{s+2a+1}
\end{align}
with equivalent norms. In the latter reference the authors show this identity for $\Omega=\mathbb{\mathbb{R}}^d$ and from there we conclude the statement in Proposition \ref{app:bv_interpolation}. \\
To avoid the abstract smoothness condition in  Theorem \ref{thm:rate_abstract} we state our theorem under a slightly stronger smoothness assumption and comment on the weaker condition in a remark afterwards. Again, we do not state bounds on the bias for the sake of brevity. 

\begin{corollary}[Convergence rates for $\mathrm{BV}$-regularization]\label{cor:bv}
Suppose Assumption \ref{ass:bv} holds true, and the true solution $f$ has smoothness 
\[ f\in \Bspace s{t_s}{t_s}\quad\text{ with } \Bn f s{t_s}{t_s}<\varrho\]     
for some $0<s<1$ and $\varrho>0$ or 
\[ f\in \BV(\Omega)\quad\text{ with } \|f\|_{\BV} <\varrho.
\]     
In the latter case we set $s=1$. 
Set $\overline{p}=\frac{2a+2}{2a+1}$ and suppose that the closure $D_F$ of $\tilde{D}_F$ in $\Bspace {-a}22$ contains an $L^{\overline{p}}(\Omega)$-ball with radius $\tau$ around $f.$ 
Suppose that $\gobs\in \Y$ satisfies \eqref{eq:det_noise_level} for $0<\delta <\varrho^{-\frac{a}{s}} \tau^\frac{s+a}{s}$ and let $\fh \in S_\alpha(\gobs)$ for some 
$\alpha>0$. 
Let $0<c_l \leq c_r$ and $1 < c_{\mathrm{D}} \leq C_{\mathrm{D}}$. 
Then there is a constant $C_r$ independent of $f, \gobs ,\varrho$, $\tau$ and $\delta$ such that either of the conditions
\[  c_l \varrho^{-\frac{a+1}{s+a}} \delta^\frac{s+2a+1}{s+a} \leq \alpha\leq c_r \varrho^{-\frac{a+1}{s+a}} \delta^\frac{s+2a+1}{s+a} \quad\text{and}\quad  c_{\mathrm{D}} \delta \leq \| \gobs - F(\fh)\|_\Y \leq C_{\mathrm{D}} \delta \] 
on the choice of $\alpha$ implies the following bounds: 
\begin{align*}
\Bn {f-\fh} {-a}22  & \leq C_r \delta\\ 
\| f-\fh\|_{L^{\overline{p}}(\Omega)} & \leq C_r \varrho^\frac{a}{s+a} \delta^\frac{s}{s+a}\\ 
\|\fh \|_{\mathrm{BV}(\Omega)} &  \leq C_r \varrho^\frac{a+1}{s+a} \delta^\frac{s-1}{s+a}
\end{align*}
\end{corollary}
\begin{proof}
We show that Assumption \ref{ass:oversmooth} is satisfied with $\XR=\BV(\Omega)$, $\X_-=\Bspace {-a}22$, $\XL=L^{\overline{p}}(\Omega)$ and $u=1$. Due to Proposition \ref{app:bv_embed} we have a continuous embedding $\BV(\Omega)\subset \Bspace {-a}22$. 
If $a=0$, then we have $\overline{p}=2$ and $\Bspace 022=L^2(\Omega)$ (see Proposition \ref{app:relation_lp_sob}). Hence we have Assumption \ref{ass:oversmooth} with $\xi=1$ in this case and the result follows from Theorem \ref{thm:rate_abstract}.\\
If $a>0$, we set
 $\xi:=\frac{a}{a+1}$. Note that $1<\overline{p}=t_0< 2$. Hence \cite[Thm.~3.4.1.(b)]{Bergh1976}, \eqref{eq:BV_interpolation} and Proposition \ref{app:relation_lp_sob} yield the following chain of continuous embeddings 
\begin{align}\label{eq:bv_loss_interpolation}
\left( \Bspace {-a}22, \BV(\Omega) \right)_{\xi,1} \subset \left( \Bspace {-a}22, \BV(\Omega) \right)_{\xi,\overline{p}}= \Bspace 0{\overline{p}}{\overline{p}} \subset L^{\overline{p}}(\mathbb{R}^d). 
\end{align}
Finally, \cite[Thm.~3.4.1.(b)]{Bergh1976} and \eqref{eq:BV_interpolation} yields
\[ \Bspace s{t_s}{t_s} = \left( \Bspace {-a}22, \BV(\Omega) \right)_{\theta_s,t_s} \subset \left( \Bspace {-a}22, \BV(\Omega) \right)_{\theta_s,\infty}.\]
Hence the smoothness condition on $f$ in the claim implies the smoothness condition in Theorem \ref{thm:rate_abstract}. Therefore, the stated result follows from Theorem \ref{thm:rate_abstract}.
\end{proof} 
\begin{remark}[Weaker smoothness condition]\label{cor:bv_smoothness_ass}
The statements in Corollary \ref{cor:bv} remain true if the smoothness assumption on $f$ is replaced by $f\in \left( \Bspace {-a}22, \BV(\Omega)\right)_{\theta_s,\infty}$ with a bound by $\varrho$ on the norm of $f$ therein. 
\end{remark}

\begin{remark}[Similarity to $\Bspace 111$-regularization]
We see that the convergence rates and also the smoothness condition for $\mathrm{BV}$-regularization equals the ones for $\Bspace 111$-regularization in Corollary \ref{cor:over_besov}. The reason for that is that the interpolation identity in \eqref{eq:BV_interpolation} holds true with $\mathrm{BV}(\Omega)$ replaced by  $\Bspace 111$. \\
Whereas for $\Bspace 111$-regularization with a norm given by wavelet coefficients we also have a convergence rate result in the non oversmoothing case $s>r$ (see \cite{MH:20}) a similar result remains open for $\mathrm{BV}$-regularization.
\end{remark}

\section{White noise}\label{sec:white_noise}
In this section we extend the tools developed in the previous sections to derive 
convergence rates for oversmoothing regularization with stochastic noise models. 

In this section we will assume that $\Omega_{\Y}\subset \mathbb{R}^d$ is 
a bounded Lipschitz domain and $\Y = L^2(\Omega_{\Y})$. We consider noise models 
of the form
\begin{align}\label{eq:white_noise_model}
\gobs = F(f) + \sigma Z,\qquad Z\in B^{-d/2}_{p',\infty}(\Omega_{\Y})
\end{align}
with a normalized noise process $Z$ and a noise level $\sigma>0$. 
Moreover, $p\in (1,\infty]$ 
is the same as in \Cref{sec:besov}, and $\frac{1}{p'}+\frac{1}{p}=1$. 
The choice of the Besov space is motivated by the fact that Gaussian white noise 
belongs to $B^{-d/2}_{p',\infty}$ almost surely (see \cite{veraar:11} for the 
$d$-dimensional torus), and to no smaller Besov spaces. 
Also point processes (i.e.\ random finite sums of 
delta-peaks) belong to $B^{-d/2}_{p',\infty}$ for $p\geq 2$ as well as 
local averages of noise processes over a finite number of detector areas. 
We will derive error bounds in terms of Besov norms of $Z$. 
The expectation of $Z$ does not necessarily have to vanish, i.e.\ $Z$ may also 
contain deterministic error components. However, to derive error bounds in 
expectation we will have to assume that the norm of $Z$ has finite moments: 
\begin{align}\label{eq:finite_moments}
\mathbb{E}\left[\|Z\|_{B^{-d/2}_{p',\infty}}^\kappa \right]<\infty\qquad \mbox{for all }
\kappa\in \mathbb{N}
\end{align}
This easily follows from much stronger large deviation inequalities (see, e.g., the 
proof of \cite[Cor.~6.5]{HM:19}), which have been shown for Gaussian white noise 
in \cite[Cor.~3.7]{veraar:11} or \cite[remark after Thm. 4.4.3]{GN:15}. 
For other noise processes 
the verification of \eqref{eq:finite_moments} may require further investigations. 

Since the Tikhonov functional in \eqref{eq:Tik_det} is not well defined 
in our setting, we formally subtract 
$\frac{1}{2}\|\gobs\|_{\Y}^2$ from $\frac{1}{2}\|\gobs-g\|_{\Y}^2$ to obtain 
the new data fidelity functional 
$\mathcal{S}_{\gobs}(g):=\frac{1}{2}\|g\|_{\Y}^2 -\langle \gobs,g\rangle$ and 
Tikhonov regularization of the form 
\begin{align}\label{eq:Tikh_white_noise}
T_\alpha(\gobs):= \argmin_{h\in \tilde{D}_F}\infbracket{\frac{1}{\alpha} 
\mathcal{S}_{\gobs}(F(h)) + \frac{1}{u}\|h\|^u_{\X_{\mathcal{R}}}}.
\end{align}
with $u\in (0,\infty)$.
Note that for $\gobs\in \Y$ we have $T_{\alpha}(\gobs) = S_{\alpha}(\gobs)$, 
but $T_{\alpha}(\gobs)$ is also well defined for white noise. More precisely, 
in the setting of the following Theorem \ref{thm:white_noise_besov}, the existence of minimizers 
in \eqref{eq:Tikh_white_noise} can be shown by the same argument as in the 
non-oversmoothing case (see \cite[Prop. 6.3]{HM:19}).

\subsection{Convergence rates}
We first study Besov penalties with $p>1$. 

\begin{theorem}[Stochastic rates for oversmoothing Besov space regularization]\label{thm:white_noise_besov}
Let  $1<p\leq 2$, $1\leq q \leq p$ and $r>0$. 
Let the data $\gobs$ be described by \eqref{eq:white_noise_model}, consider Tikhonov 
regularization in the form \eqref{eq:Tikh_white_noise}, and assume that 
$\|\cdot\|_{\XR}$ in  \eqref{eq:Tikh_white_noise} is equivalent to $\Bn \cdot rpq$.
Suppose the true solution $f\in \tilde{D}_F$ has regularity $s\in (0,r]$ with norm bound $\varrho>0$ 
in the sense of \eqref{eq:besov_regularity} in Corollary \ref{cor:over_besov} 
or \eqref{eq:over_besov_smoothness} in Remark \ref{rem:over_besov_smoothness}. 
In addition to Assumption \ref{ass:besov} suppose that $F$ satisfies the one-sided Lipschitz condition 
\begin{align}\label{eq:Lipschitz}
\|F(f_1)- F(f_2)\|_{\Bospace {a+r}pq (\Omega_{\Y})}   
\leq \widetilde{M}_2 \|f_1-f_2\|_{\Bospace {r}pq (\Omega)}
\quad\text{ for all } f_1,f_2 \in \tilde{D}_F
\end{align} 
and that $\frac{d}{2}<a+r$. 
Let $\overline{p}:=\frac{2p(a+r)}{2a+pr}$ and assume that the closure $D_F$ of $\tilde{D}_F$ in $\Bspace {-a}22$ contains a $\Bspace 0{\overline{p}}{\overline{p}}$-ball with radius $\tau>0$ around $f$. 

Then there is a a-priori parameter choice rule $\alpha=\alpha(\sigma,\varrho)$  (specified in \eqref{eq:choise_rule_white_noise})
such that there exists a constant 
$C_r>0$ such that the reconstruction error with $\fh\in T_\alpha(\gobs)$ satisfies the bounds
\begin{subequations}\label{eq:besov_bounds_white_noise}
\begin{align}
\Bn {f-\fh} {-a}22  & \leq C_r (1+N^\eta)\rho^{\frac{d/2}{s+a+d/2}}\sigma^{\frac{s+a}{s+a+d/2}}  \\ 
\left\| f-\fh\right\|_{L^{\overline{p}}} & \leq C_r (1+N^\eta)\varrho^{\frac{a+d/2}{s+a+d/2}}\sigma^{\frac{s}{s+a+d/2}} \\ 
\left\|\fh\right\|_{\XR}&  \leq C_r (1+N^\eta)\varrho^{\frac{r+a+d/2}{s+a+d/2}}\sigma^{\frac{s-r}{s+a+d/2}}
\end{align}
\end{subequations}
for all $0< \sigma <\varrho^{-\frac{a+d/2}{s}}\tau^\frac{s+a+d/2}{s} $ 
with $N:=\|Z\|_{B^{-d/2}_{p',\infty}}$,   
and  $\eta:=\frac{(a+r)u}{(a+r+d/2)u/2-d/2}$. In particular, 
if \eqref{eq:finite_moments} holds true, then 
\begin{align}\label{eq:exp_rate}
\mathbb{E}\left[\left\| f-\fh\right\|_{L^{\overline{p}}}^\kappa\right]^{\frac{1}{\kappa}} 
= \mathcal{O}\left(\varrho^{\frac{a+d/2}{s+a+d/2}}\sigma^{\frac{s}{s+a+d/2}}\right)
\qquad \mbox{as } \sigma \to 0 \quad
\mbox{for all }\kappa\geq 1. 
\end{align}
\end{theorem}

\begin{proof}
As in \Cref{sec:besov} we set $\X_- = B^{-a}_{2,2}(\Omega)$. If $a>0$ then we have 
$(\X_-,\XR)_{\frac{a}{a+r},\overline{p}}
\subset B^0_{\overline{p},\overline{p}}(\Omega) \subset L^{\overline{p}}(\Omega)=\XL$ 
with continuous embeddings due to $p\leq 2$ 
and $q\leq p$  (see \eqref{eq:over_besov_loss_inter} and Remark \ref{rem:over_Lp_loss}). If $a=0$, then $\X_-=\XL=L^2(\Omega)$.
We choose \[ {t=C_L^{-\frac{a+r}{s}}(\sigma / \varrho)^\frac{a+r}{s+a+d/2}},
\]  
and from Lemma \ref{lem:aux_elements} with $\theta = \frac{a+s}{a+r}$ we obtain  
\begin{align}\label{eq:xl_white_noise}
 \|f-f_t\|_{\XL} \leq \varrho^\frac{a+d/2}{s+a+d/2} \sigma^\frac{s}{s+a+d/2} <\tau.
 \end{align}
 Hence $f_t\in D_F$, and by definition $\fh\in T_{\alpha}(\gobs)$ implies 
\[
\frac{1}{\alpha}\mathcal{S}_{\gobs}(\gh) + \frac{1}{u}\|\fh\|_{_{\X_{\mathcal{R}}}}^u
\leq \frac{1}{\alpha}\mathcal{S}_{\gobs}(g_t) + \frac{1}{u}\|f_t\|_{_{\X_{\mathcal{R}}}}^u
\] 
with $g_t:= F(f_t)$ and $\gh:=F(\fh)$. 
Adding 
$\frac{1}{2\alpha}\|g_t-\gh\|_{\Y}^2- \frac{1}{\alpha}\mathcal{S}_{\gobs}(\gh)$ to this equation 
yields
\begin{align}
&\frac{1}{2\alpha} \|g_t-\gh\|_{\Y}^2 + \frac{1}{u}\|\fh\|^u_{\X_{\mathcal{R}}}
\nonumber\\
\label{eq:aux_white_noise}
&\leq  \frac{1}{2\alpha} \|g_t-\gh\|_{\Y}^2+ \frac{1}{\alpha}\mathcal{S}_{\gobs}(g_t)  
-\frac{1}{\alpha}\mathcal{S}_{\gobs}(\gh)
+ \frac{1}{u}\|f_t\|_{_{\X_{\mathcal{R}}}}^u \\
&= \frac{1}{\alpha}\left\langle \sigma Z,\gh-g_t\right\rangle 
+ \frac{1}{\alpha}\left\langle F(f)-g_t,\gh-g_t\right\rangle 
+ \frac{1}{u}\|f_t\|_{_{\X_{\mathcal{R}}}}^u.\nonumber 
\end{align}  
The first term on the left hand side is estimated using the Besov space interpolation 
\begin{align}\label{eq:interpolation_data_white_noise}
\left(B^0_{p,2}(\Omega_{\Y}), B^{a+r}_{p,q}(\Omega_{\Y})\right)_{d/(2a+2r),1} =  
B^{d/2}_{p,1}(\Omega_{\Y}),
\end{align}
the Lipschitz condition \eqref{eq:Lipschitz},  and the continuity of the embedding 
$B^0_{2,2}(\Omega_{\Y}) = \Y \hookrightarrow B^0_{p,2}(\Omega_\Y) $:
\begin{align}\label{eq:Z_estimate} \nonumber
\frac{1}{\alpha}\left\langle \sigma Z,\gh-g_t\right\rangle
&\leq \frac{1}{\alpha}\|\sigma Z\|_{B^{-d/2}_{p',\infty}} \left\|g_t-\gh\right\|_{B_{p,1}^{d/2}}\\ \nonumber
&\leq  c_1\frac{\sigma N}{\alpha}
\left\|g_t-\gh\right\|_{B^0_{p,2}}^{1-\frac{d}{2a+2r}}
\left\|g_t-\gh\right\|_{B_{p,q}^{a+r}}^{\frac{d}{2a+2r}}\\
&\leq c_2 \frac{\sigma N}{\alpha} 
\left\|g_t-\gh\right\|_{\Y}^{1-\frac{d}{2a+2r}}
\left\|f_t-\fh\right\|_{\XR}^{\frac{d}{2a+2r}}\\  \nonumber
&= \left( c_2  \sigma N \alpha^{-\frac{a+r+d/2}{2a+2r}} \right) \left( \frac{1}{\alpha}{\left\|g_t-\gh\right\|_{\Y}^2} \right)^\frac{a+r-d/2}{2a+2r} \left(\left\|f_t-\fh\right\|_{\XR}^u \right)^{\frac{d}{u(2a+2r)}} 
\end{align}
with $c_2$ depending on $c_1$, the embedding constant, $\widetilde{M}_2$, the constant in the equivalence of $\Bn {\cdot} rpq$ and $\|\cdot \|_{\XR}$. 
Now Young's inequality $xyz \leq \frac{1}{\theta}x^\theta +\frac{1}{\mu}y^{\mu}+ \frac{1}{\nu} z^\nu$ for $ \frac{1}{\theta}+\frac{1}{\mu}+\frac{1}{\nu}=1 $ with $\eta=\frac{u(2a+2r)}{(a+r+d/2)u-d}$, $\mu:=\frac{2a+2r}{a+r-d/2}$ and $\nu:=\frac{u(2a+2r)}{d}$ and the elementary inequality $(x+y)^u\leq 2^{u-1}(x^u+y^u)$ yield
\begin{align*}
\frac{1}{\alpha}\left\langle \sigma Z,\gh-g_t\right\rangle
&\leq  c_3  \left( \sigma N \alpha^{-\frac{a+r+d/2}{2a+2r}} \right)^\eta + \frac{1}{8\alpha}\left\|g_t-\gh\right\|_{\Y}^2+ \frac{1}{2u} 2^{1-u} 
\left\|f_t-\fh\right\|_{\XR}^u\\
&\leq  c_3  \left( \sigma N \alpha^{-\frac{a+r+d/2}{2a+2r}} \right)^\eta + \frac{1}{8\alpha}\left\|g_t-\gh\right\|_{\Y}^2+ \frac{1}{2u} 
\|\fh\|_{\XR}^u + \frac{1}{2u} 
\left\|f_t\right\|_{\XR}^u
\end{align*}
with a constant $c_3$ that depends on $c_2, u,\eta, \mu$ and  $\nu$.
The second and third summand on the right hand side can be absorbed in the left hand side of 
\eqref{eq:aux_white_noise}. 
The second term on the right hand side \eqref{eq:aux_white_noise} is estimated by 
\begin{align*}
\frac{1}{\alpha}\left\langle F(f)-g_t,\gh-g_t\right\rangle
&\leq \frac{1}{\alpha}\|F(f)-g_t\|_{\Y} \|\gh-g_t\|_{\Y} \\
&\leq \frac{M_2}{\alpha} \|f-f_t\|_{\X_-} \|\gh-g_t\|_{\Y}\\
&\leq \frac{4M_2^2}{\alpha}\|f-f_t\|_{\X_-}^2 + 
\frac{1}{8\alpha}  \|\gh-g_t\|_{\Y}^2,
\end{align*}
and the second term can be absorbed in the left hand side of \eqref{eq:aux_white_noise}. 
Altogether we have shown that 
\begin{align}
\frac{1}{4\alpha} \|g_t-\gh\|_{\Y}^2 + \frac{1}{2u}\|\fh\|^u_{\X_{\mathcal{R}}}
&\leq c_3  \left( \sigma N \alpha^{-\frac{a+r+d/2}{2a+2r}} \right)^\eta
+ \frac{3}{2u} \|f_t\|_{\XR}^u + \frac{4M_2^2}{\alpha}\|f-f_t\|_{\X_-}^2 \nonumber\\
\label{eq:white_noise_aux2}
&\leq c_3  \left( \sigma N \alpha^{-\frac{a+r+d/2}{2a+2r}} \right)^\eta
+ c_4 \varrho^\frac{2u(a+r)}{(2-u)s+2a+ur} \alpha^\frac{u(s-r)}{(2-u)s+2a+ur}\\
&\leq (c_3+c_4) (1+N^\eta) \varrho^\frac{u(a+r+d/2)}{s+a+d/2}  \sigma^\frac{u(s-r)}{s+a+d/2} \nonumber
\end{align} 
using Lemma \ref{lem:aux_elements}, the choice of  $t$ and the parameter choice rule 
\begin{align}\label{eq:choise_rule_white_noise}
 \alpha=c_\alpha \varrho^{-\frac{u(a+r+d/2)-d}{s+a+d/2}} \sigma ^{\frac{(2-u)s+2a+ur}{s+a+d/2}} 
\end{align}
for $\alpha$ with a constant $c_\alpha$. Here the constant $c_4$ depends on $M_2$, $C_L$, $u$, $a$, $s$, $r$ and $c_\alpha$.   \\
This shows on the one hand that 
\begin{align*}
\left\|f_t -\fh\right\|_{\X_-} 
\leq  M_1\left\|g_t-\gh\right\|_{\Y} 
\leq c_5   (1+N^\eta) \varrho^{\frac{d/2}{s+a+d/2}}\sigma^{\frac{s+a}{s+a+d/2}}.
%
\end{align*}
with $c_5$ depending on  $M_1$, $c_3$ and $c_4$.  
where we use the choice of $\alpha$ once again. This finishes the proof if $a=0$. 
On the other hand, \eqref{eq:white_noise_aux2} and Lemma \ref{lem:aux_elements} 
implies 
\begin{align*}
\|f_t -\fh\|_{\X_R} 
\leq \|f_t\|_{\X_R} + \|\fh\|_{\X_R} \leq c_6 (1+N^\eta) \varrho^\frac{a+r+d/2}{s+a+d/2}  \sigma^\frac{s-r}{s+a+d/2} 
\end{align*}
with $c_6$ depending on $C_L$,$u$, $c_3$, $c_4$.
Putting both estimates together and using the interpolation and embedding results from 
the very beginning of this proof, we obtain 
\begin{align*}
\|f_t-\fh\|_{\XL} &\leq \|f_t-\fh\|_{\X_-}^{\frac{r}{a+r}}\|f_t-\fh\|_{\X_R}^{\frac{a}{a+r}}
\leq c_7 (1+N^\eta) 
\varrho^{\frac{a+d/2}{a+s+d/2}}\sigma^{\frac{s}{a+s+d/2}}
\end{align*} 
with $c_7$ depending on $c_5$ and $c_6.$
Together with \eqref{eq:xl_white_noise} we wind up with
\[
\|f-\fh\|_{\XL} \leq \|f-f_t\|_{\XL}+ \|f_t-\fh\|_{\XL} 
\leq (c_7+1)(1+N^\eta) \varrho^{\frac{a+d/2}{a+s+d/2}}\sigma^{\frac{s}{s+a+d/2}}.
\]
\end{proof}


In our noise model \eqref{eq:white_noise_model} we excluded the case $p=1$, i.e.\ $p'=\infty$, 
since $Z\notin B^{-d/2}_{\infty,\infty}$ almost surely (see \cite{veraar:11}). 
However, the interesting case $p=q=1$ can be treated if we impose an additional one-sided 
Lipschitz condition on $F$:  

\newcommand{\newp}{p}

\begin{theorem}[Stochastic rates for oversmoothing regularization with BV or $B^r_{1,1}$ penalties]\label{thm:white_noise_besov1}
For data $\gobs$ described by \eqref{eq:white_noise_model} with 
$p$ defined below, 
consider Tikhonov regularization of the form \eqref{eq:Tikh_white_noise} with either $\|\cdot\|_\XR=\|\cdot\|_{\BV}$ or $\|\cdot\|_\XR=\|\cdot\|_{B^r_{1,1}}$ for some 
$r>\max(0,\frac{d}{2}-a)$. For $\BV$ we set $r=1$ and assume that 
$a>\frac{d}{2}-1$. 
Suppose the true solution has regularity 
\[ f\in \Bspace s{t_s}{t_s}\quad\text{ with } \Bn f s{t_s}{t_s}<\varrho\] 
for $s\in (0,r]$ and $t_s = \frac{2a+2r}{s+2a+r}$.  
For $\|\cdot\|_\XR=\|\cdot\|_{\BV}$ and $s=1$ we assume 
$f\in \BV(\Omega)$ and $\|f\|_{\BV}\leq \rho$. 
In addition to Assumption \ref{ass:bv} or \ref{ass:besov}, respectively, 
suppose that there exists $e\in (0,a+r-d/2)$ such that $F$ satisfies the one-sided  Lipschitz condition 
\begin{align}\label{eq:Lipschitz_1}
\|F(f_1)- F(f_2)\|_{\Bospace {a+r-e}{\newp}{\newp} (\Omega_\Y)}   
\leq \widetilde{M}_2 
\|f_1-f_2\|_{\Bospace {r-e}{\newp}{\newp}(\Omega)}\quad\text{ for all } f_1,f_2 \in \tilde{D}_F
\end{align} 
with $\newp:=\frac{a+r}{a+r-e/2}$. Let $\overline{p}:=\frac{2a+2r}{2a+r}$ and assume that the closure $D_F$ of $\tilde{D}_F$ in $\Bspace {-a}22$ contains a $\Bspace 0{\overline{p}}{\overline{p}}$-ball with radius $\tau>0$ around $f$. 

Then there is a a-priori parameter choice rule $\alpha=\alpha(\sigma,\varrho)$  (specified in \eqref{eq:choise_rule_white_noise})
such that there exists a constant 
$C_r>0$ such that the reconstruction error with $\fh\in T_\alpha(\gobs)$ satisfies all three bounds in \eqref{eq:besov_bounds_white_noise} for all $\sigma$ as in Theorem \ref{thm:white_noise_besov}
with $N:=\|Z\|_{B^{-d/2}_{\newp',\infty}}$ and $\eta$ as in Theorem \ref{thm:white_noise_besov}. 
Under the assumption \eqref{eq:finite_moments} we also have \eqref{eq:exp_rate}. 
\end{theorem}

\begin{proof}
The proof follows along the lines of the proof of Theorem \ref{thm:white_noise_besov}, 
we just have to replace \eqref{eq:Z_estimate} as follows: Note that $1<\newp<2$. 
The starting point is 
\begin{align}\label{eq:Zestimate1_start}
\frac{1}{\alpha}\left\langle \sigma Z,\gh-g_t\right\rangle
&\leq \frac{1}{\alpha}\|\sigma Z\|_{B^{-d/2}_{\newp',\infty}} \left\|g_t-\gh\right\|_{B_{\newp,1}^{d/2}}
= \frac{\sigma N}{\alpha} \left\|g_t-\gh\right\|_{B_{\newp,1}^{d/2}}
\end{align}
We replace \eqref{eq:interpolation_data_white_noise} by
\begin{align*}
\left(B^0_{\newp,2}(\Omega_{\Y}), B^{a+r-e}_{\newp,\newp}(\Omega_{\Y})\right)_{d/(2a+2r-2e),1} 
=  B^{d/2}_{\newp,1}(\Omega_{\Y}),
\end{align*}
and use the continuity of the embedding $B^{0}_{2,2}(\Omega_{\Y})=\Y\hookrightarrow  B^{0}_{t,2}(\Omega_{\Y})$ and \eqref{eq:Lipschitz_1} to obtain 
\begin{align}\label{eq:inter_ineq_white_noise1}
\nonumber
\left\|g_t-\gh\right\|_{B^{d/2}_{\newp,1}} 
&\leq c_1 \left\|g_t-\gh\right\|_{B^0_{\newp,2}}^{1-\frac{d}{2a+2r-2e}}    \left\|g_t-\gh\right\|_{B^{a+r-e}_{\newp,\newp}} ^\frac{d}{2a+2r-2e} \\ 
&\leq c_2 \left\|g_t-\gh\right\|_{\Y}^{1-\frac{d}{2a+2r-2e}}\left\|f_t-\fh\right\|_{B^{r-e}_{\newp,\newp}}^\frac{d}{2a+2r-2e} 
\end{align}
with $c_2$ depending on $c_1$, the embedding constant and $\tilde{M}_2$.
To estimate the second factor on the right hand side we use the interpolation identity 
\begin{align}\label{eq:mixed_inter_white_noise1}
B^{r-e}_{\newp,\newp}(\Omega)= \left( B^{-a}_{2,2}(\Omega) , 
\XR 
\right)_{\frac{a+r-e}{a+r},\newp}
\end{align}
(note that $-a<r-e$), which follows from 
Proposition \ref{app:bv_interpolation} or \cite[2.4.3]{Triebel2010}, respectively. 
Together with Assumption \ref{ass:bv} resp. \ref{ass:besov} we obtain 
\begin{align*}
\left\|f_t-\fh\right\|_{B^{r-e}_{\newp,\newp}} &  \leq c_3 \left\|f_t-\fh\right\|_{B^{-a}_{2,2}}^\frac{e}{a+r} \left\|f_t-\fh\right\|_{\XR}^\frac{a+r-e}{a+r}  \\ 
& \leq c_4 \left\|g_t-\gh\right\|_{\Y}^\frac{e}{a+r} \left\|f_t-\fh\right\|_{\XR}^\frac{a+r-e}{a+r}
\end{align*} 
with $c_4$ depending on $c_3$ and $M_1$. 
Inserting into \eqref{eq:inter_ineq_white_noise1} and then into \eqref{eq:Zestimate1_start}  yields the inequality 
\[ \frac{1}{\alpha}\left\langle \sigma Z,\gh-g_t\right\rangle \leq c_5 \frac{\sigma N}{\alpha} \left\|g_t-\gh\right\|_{\Y}^{1-\frac{d}{2a+2r}} \left\|f_t-\fh\right\|_{\XR}^\frac{d}{2a+2r},  \] 
which replaces \eqref{eq:Z_estimate}. Here $c_5$ depends on $c_4$ and $c_2$. 
The rest of the proof can be copied from the proof of \Cref{thm:white_noise_besov}.
\end{proof}

\begin{remark}[minimax optimality]
It can be shown as in \cite[Prop. 6.6]{HM:19} that the error bound in \eqref{eq:exp_rate} 
is optimal in a minimax sense. 
\end{remark}

\begin{remark}[duality]
In view of the fact that the dual of Besov spaces $\Bspace {-s}{p'}{q'}$ for 
$s\in \mathbb{R}$, $p,q\in (1,\infty)$ on a smooth, bounded domains $\Omega$ 
is given by the spaces 
$\tildeB spq :=\{f\in \Bospace spq (\mathbb{R}^d):\supp f\subset \overline{\Omega}\}$ (see \cite[Thm.~4.8.2]{Triebel1978}), it may appear more natural to impose the assumptions 
\eqref{eq:Lipschitz} and \eqref{eq:Lipschitz_1} in these spaces. 
(Note that if $\Omega_{\Y}=\Omega$, $F$ is linear and self-adjoint on $L^2(\Omega)$ 
with a bijective continuous extension to $\Bospace {-a}22(\Omega)\to L^2(\Omega)$ 
such that Assumption \ref{ass:besov} holds true, then 
$F:L^2(\Omega)\to \tildeB a22$ is also bijective by duality.) 
However, the spaces $\tildeB spq $ are closed subspaces of $\Bospace spq(\Omega)$, 
which can be written as nullspaces of certain trace operators, except for smoothness
indices $s$ with $s-\frac{1}{p}\in \mathbb{N}_0$ at which the number of well-defined traces changes (see \cite[Thms.~4.3.2/1 and 4.7.1]{Triebel1978}). 
Therefore, the given formulations of \eqref{eq:Lipschitz} and \eqref{eq:Lipschitz_1} are more general, and boundary conditions can be incorporated 
in the domain $\tilde{D}_F$ of $F$.  
\end{remark}

\begin{remark}[special case $f\in \BV$]
Theorem \ref{thm:white_noise_besov1} for 
$\|\cdot\|_{\XR}=\|\cdot\|_{\BV}$ and $f\in \BV$ improves 
the rate in \cite{dAM:20} obtained for a different estimator 
by eliminating logarithmic factors in the noise level. Furthermore, we do not need to assume the existence of a wavelet-vaguelette decomposition of the forward operator. 
\end{remark}

\begin{remark}[implications for regression]
Our setting includes the case $F=I$ corresponding to regression 
problems. We discuss two particular cases:
\begin{itemize}
\item If we choose Besov wavelet norms with $p=q=1$ as in \Cref{sec:Besov_wavelet}, 
then the minimization of the Tikhonov functional splits into a family of minimization 
problems for each wavelet coefficients resulting in 
\emph{soft thresholding} or \emph {wavelet shrinkage} estimators with level-dependent 
threshold.  
Such estimators have been studied extensively in mathematical statisics 
(see, e.g., \cite{DJ:98}). 
\item For $\|\cdot\|_{\XR} =\|\cdot\|_{\BV}$ we obtain $\BV$-denoising. 
Here our assumption $a\geq \frac{d}{2}-1$ is only satisfied for $d=1$. 
In this case Theorem \ref{thm:white_noise_besov1} shows optimal 
$L^{\overline{p}}$-convergence rates of this estimator for functions with 
Besov smoothness $\leq 1$ (see also \cite{MvdG:97}). 
In higher dimensions convergence rates of 
a multiresolution estimator for $\BV$ functions were established in \cite{dALM:21}. 
\end{itemize}
\end{remark}
\subsection{Numerical experiments for a parameter identification problem}
We confirm the theoretical results in Theorems \ref{thm:white_noise_besov} and $\ref{thm:white_noise_besov1}$ by numerical experiments for the nonlinear identification of $c$ in the elliptic 
boundary value problem  
\begin{align}\label{eq:bvp_1d}
\begin{aligned}
&- u^{\prime\prime} + c u = \varphi &&\mbox{in }(0,1),\\
&u(0)=u(1)=1.
\end{aligned}.
\end{align}
The forward operator in the function space setting is 
$F(c):=u$ for the fixed smooth right hand side $\varphi$. For this problem the verification of Assumption \ref{ass:besov} with $a=2$ is discussed in \cite[Ex.~2.8, Lem.~2.9]{HM:19}.  
The experiments are carried out in the same setup as in \cite{MH:20} where more 
details on the implementation can be found.  
We added independent $N(0,\tilde{\sigma}^2)$-distributed random variables to $n=2^{10}$ equidistant measurement points as a discrete approximation of Gaussian white noise 
on $[0,1]$ with $\sigma = \frac{\tilde{\sigma}}{\sqrt{n}}$. 

The true coefficient $c^{\mathrm{jump}}$ is given by a piecewise smooth function with finitely many jumps. For each noise level $\tilde{\sigma}$ we drew $10$ data sets and took the average of the reconstruction errors.  

The regularization parameter $\alpha$ was chosen according to the rule \eqref{eq:choise_rule_white_noise} with $c_\alpha$ chosen optimally for medium 
value of $\tilde{\sigma}$. Of course, in practice $\alpha$ would have to be chosen 
in a completely data-driven manner, e.g.\ by the Lepski{\u\i} balancing principle, 
but this is not in the scope of this paper. 

\begin{example}[$p=2,q=1$]\label{ex:over_2}
First, we use as penalty the norm (with power $u=1$) on the Besov space $\Bospace 221((0,1))$ given by the $\bspace 221$-norm of wavelet coefficients with respect to Daubechies wavelets of order $7$. 
According to Remark \ref{rem:over_2}, smoothness of the solution $c^{\mathrm{jump}}$ 
is then measured in the scale $\Bospace {s}2\infty((0,1))$, and in this scale the maximal smoothness index of $c^{\mathrm{jump}}$ is $s=\frac{1}{2}$, i.e.\ 
$c^{\mathrm{jump}}\in \Bospace {1/2}2\infty ((0,1))$ (see \cite[Ex.30]{HM:19}).  
In Figure \ref{fig:plot} we see a good agreement of the reconstruction error in the numerical experiment with the predicted rate $\mathcal{O}(\sigma^{1/6})$ measured in the $L^2$-norm. 
\end{example}

\begin{example}[$p=q=1$]\label{ex:over_1}
Now we use the $\bspace 211$ norm on $db7$ wavelet coefficients norm as 
penalty term.
%
As $a=r=2$, we have $\overline{p}=\frac{4}{3}$. As in \cite{HM:19} one shows that 
$c^{\mathrm{jump}}$ belongs to $\Bospace {s}{t_s}{t_s}((0,1)) $ for $s< \frac{6}{7}$. 
Therefore, Corollary \ref{cor:over_besov} and Remark \ref{rem:over_Lp_loss} predict the rate $\mathcal{O}(\sigma^e)$ for all $e<\frac{12}{47}$ measured in the $L^\frac{4}{3}$-norm. In Figure \ref{fig:plot} we see a good agreement with the reconstruction error in the numerical experiment.
\end{example}
\begin{figure}[ht]
\includegraphics[width=\textwidth]{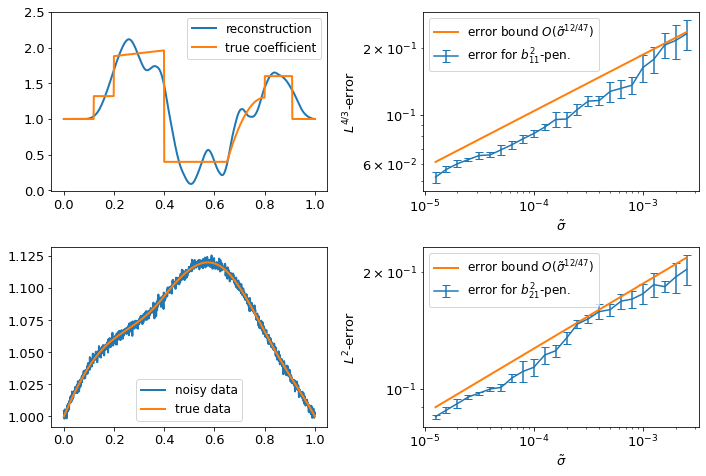} 
\caption{\emph{Left top:} the true coefficient $c^{\mathrm{jumps}}$ with jumps in the boundary value problem \eqref{eq:bvp_1d} together with the reconstruction for $\bspace 211$-penalization at noise level $\tilde{\sigma}=2.51\cdot 10^{-3}$. \emph{Left bottom:}
 Corresponding noisy data together with $F(c^{\mathrm{jumps}})$. \emph{Right top:} Averaged reconstruction error and its standard derivation using $\bspace 221$-penalization, the rate $\mathcal{O}(\tilde{\sigma}^{1/6})$ in the $L^2$-norm predicted by Theorem \ref{thm:white_noise_besov} (see Example \ref{ex:over_2}).
 \emph{Right bottom:} Reconstruction error using $\bspace 211$-penalization, the rate $\mathcal{O}(\tilde{\sigma}^{12/47})$ in the $L^{4/3}$-norm predicted by Theorem \ref{thm:white_noise_besov1} (see Example \ref{ex:over_1}).}
\label{fig:plot} 
\end{figure}

\section{Discussion and conclusions}
We end this paper by a summary of our results and a comparison to non-oversmoothing regularization theory. 
Until recently the oversmoothing case in variational regularization theory has been considered more difficult to analyze due to the failure of the tools developed for the non-oversmoothing case so far, which are usually based on some type of source condition. 
The analysis of this paper, inspired by a series of 
recent papers discussed in the introduction, suggests that on the contrary oversmoothing 
may be considered the easier case. 
The theory is now more complete in many respects than the theory of non-oversmoothing Banach space regularization as the following examples demonstrate: 
\begin{itemize}
\item For oversmoothing Banach space regularization, in contrast to the non-oversmoothing case, convergence rate results always remain valid if the norm in the penalty term 
is replaced by an equivalent norm. 
\item So far the analysis of non-oversmoothing Besov space penalization 
(see \cite{HM:19,Miller2021,MH:20,WSH:20}) is restricted 
to certain choices of the Besov norm indices $r$, $p$ and $q$ and the norm power $u$, 
whereas Corollary \ref{cor:over_besov} with the generalization in Remark \ref{rem:over_besov_smoothness} only assumes $r>0$.
\item We are not aware of a convergence rate analysis of BV regularization for the 
case that the solution belongs to a smoothness class which is smaller than $\BV$. 
(The case that the solution smoothness is exactly BV has been analyzed in \cite{dAM:20} 
in a statistical setting.) 
In contrast, Corollary \ref{cor:bv} provides optimal convergence rates for BV regularization if 
the solution only belongs to smoothness classes larger than $\BV$. 
\end{itemize}

On the other hand, an analysis of exponentially smoothing forward operators 
and other operators not satisfying a two-sided Lipschitz condition 
is still missing so far for oversmoothing Banach space regularization. 
Moreover, more flexibility in the choice of the loss function would be desirable 
both for the oversmoothing and the non-oversmoothing case, 
to allow for natural or desirable norms and for comparisons of different methods. 

\appendix 

\section{Tools from abstract interpolation theory}
We first characterize the second part of Assumption \ref{ass:oversmooth}:

\begin{proposition}[Interpolation inequality (see {\cite[Sec.~3.5, Thm.~3.11.4]{Bergh1976}})] \label{lem:intrerpolation_inequality}
Suppose $\XR$,$\XL$ and $\XA$ are quasi-Banach spaces with continuous embeddings $\XR\subset \XL\subset \XA$ and $\xi\in (0,1).$ Then the following statements are equivalent 
\begin{enumerate}
\item $\XL$ continuously embeds into $\left(\XA, \X \right)_{\xi,1}$. 
\item There exists a constant $c>0$ such that 
\[ \|f \|_\XL \leq c \| f\|_\XA^{1-\xi} \cdot \|f\|_\X^\xi \quad\text{for all } f \in \X. \]  
\end{enumerate}
\end{proposition}

\begin{proposition}\label{app:interpolation_limiting}
Let $\X$ and $\XA$ be quasi-Banach spaces with a continuous embedding $\X\subset \XA$.
Then have $\X\subset \left( \XA ,\X \right)_{1,\infty}$ with  embedding constant equal to $1$. 
\end{proposition}
\begin{proof}
Let $f\in \X$. Then we insert $h=f$ in the $K$-functional \eqref{eq:K_functional_embedded} to obtain \[ K(t,f)\leq t \|f\|_\X \quad \text{ for all }  t>0 .\] Hence $f\in \left( \XA ,\X \right)_{1,\infty}$ with $\|f\|_{\left( \XA ,\X \right)_{1,\infty}}\leq \|f\|_\X$. 
\end{proof}
\begin{proposition}[Reiteration] \label{prop:reiteration} Let $\X$ and $\XA$ be quasi-Banach spaces with a continuous embedding $\X\subset \XA$ and let $0< \xi < \theta \leq 1$.  Then 
\begin{align*}
 \left(\XA,\X\right)_{\xi,1}= \left( \XA , \left(\XA,\X \right)_{\theta,\infty}\right)_{\frac{\xi}{\theta},1}  
\end{align*}
with equivalent quasi-norms.
\end{proposition} 
\begin{proof}
In the notation of \cite[Def.~3.5.1]{Bergh1976} we have that $\XA$ is of class $\mathcal{C}\left(0, \left(\XA,\X\right)\right)$. Moreover,  $\left(\XA,\X \right)_{\theta,\infty}$ is of class $\mathcal{C}\left(\theta, \left(\XA,\X\right)\right)$. If $\theta<1$ this is due to \cite[Thm.~3.11.4]{Bergh1976}). For $\theta=1$ the definition yields that $\left(\XA,\X \right)_{1,\infty}$ is of class  $\mathcal{C}_K\left(1, \left(\XA,\X\right)\right)$ (see  \cite[Def.~3.5.1]{Bergh1976}). Moreover, from Proposition \ref{app:interpolation_limiting} we see 
\[ \|f\|_{\left(\XA,\X \right)_{1,\infty}}\leq \|f\|_\X \leq t\inv \max\{ \|f\|_\XA ,t \|	f\|_\X \}. \] 
Hence $\left(\XA,\X \right)_{1,\infty}$ is of class $\mathcal{C}_J\left(1, \left(\XA,\X\right)\right)$ (see again \cite[Def.~3.5.1]{Bergh1976}). 
Therefore, the result follows from the reiteration theorem  \cite[Thm.~3.11.5]{Bergh1976}.
\end{proof}

\section{Properties of Besov spaces}\label{sec:app_besov}
As elsewhere let $\Omega\subset\mathbb{R}^d$ be a bounded Lipschitz domain. 
We first review the relations of Besov spaces to $L^p$-spaces and to the Sobolev spaces 
$W^s_p(\Omega)$, $s\geq 0$, $p\in (1,\infty)$. 
Recall that for $s\in \mathbb{N}_0$ Sobolev norms are given by 
$\|f\|_{W^s_p}^p = \sum_{\alpha\in \mathbb{N}_0^d,|\alpha|\leq s}\|D^{\alpha}f\|_{L^p}^p$. For non-integer $s$ these spaces are also called Sobolev-Slobodeckij spaces, 
and for $p=2$ they coincide with the $H^s_2(\Omega)$ spaces defined on $\mathbb{R}^d$ via 
Fourier transform for all $s\in\mathbb{R}$. 
\begin{proposition}[Embeddings with $L^p$ and Sobolev spaces] \label{app:relation_lp_sob}
\begin{enumerate}
\item Let $p\in (1,\infty)$. Then we have continuous embeddings
\begin{align}\label{eq:besov_lp_pgeq1}
 \Bspace 0p{\min\{p,2\}} \subset L^p(\Omega)\subset \Bspace 0p{\max\{p,2\}}.
 \end{align}
For $p=1$ the following continuous embeddings hold true:
\begin{align}\label{eq:besov_l1} 
\Bspace 011 \subset L^1(\Omega)\subset \Bspace 01\infty.
\end{align} 
\item $\Bspace spp=W^s_p(\Omega)$ with equivalent norms for all 
$0<s\notin \mathbb{N}$ and $p\in (1,\infty)$, and in case of $p=2$ for all 
$s\in\mathbb{R}$. 
\item Let $p_0,p_1,q_0,q_1\in (0,\infty]$ and $-\infty<s_1<s_0<\infty$. Then 
\begin{align}\label{eq:besov_embed_q_eq}
\Bspace{s_0}{p_0}{q_0}\subset \Bspace{s_1}{p_1}{q_0}
\quad \mbox{if}\quad 
s_0-\frac{d}{p_0} = s_1-\frac{d}{p_1}
\end{align}
and 
\begin{align}\label{eq:besov_embed_q_neq}
\Bspace{s_0}{p_0}{q_0}\subset \Bspace{s_1}{p_1}{q_1}
\quad \mbox{if}\quad 
s_0-\frac{d}{p_0} > s_1-\frac{d}{p_1}.
\end{align}
\item Let $s\in \mathbb{R}$ and $p,q_0,q_1\in (0,\infty]$. Then 
\begin{align}\label{eq:besov_embed_sp_fixed}
\Bspace sp{q_0} \subset \Bspace sp{q_1}
\quad \mbox{if}\quad 
q_0\leq q_1
\end{align}
\end{enumerate}
\end{proposition}
\begin{proof}
First note that as all occurring spaces on bounded Lipschitz domains in $\mathbb{R}^d$ are defined by restriction of the respective spaces on $\mathbb{R}^d$ it suffices to prove the assertions for $\Omega=\mathbb{R}^d$.  
\begin{enumerate}
\item
Let $F^s_{p,q}(\Omega)$ be the function spaces defined in 
\cite[2.3.1.~Def.~2(ii)]{Triebel2010} for $\Omega=\mathbb{R}^d$.   
By \cite[3.2.4.(3)]{Triebel2010} 
we have continuous embeddings 
\begin{align}\label{eq:FandBconnection}
\Bspace sp{\min\{p,q\}}\subset F^s_{p,q}(\Omega)\subset \Bspace sp{\max \{p,q\}} \quad \text{for all} \quad p,q\in [1,\infty).
\end{align}
With this the embeddings for $p>1$ follow from the identity $F^0_{p,2}(\Omega)=L^p(\Omega)$ with equivalent norms. The latter identity can be found in \cite[2.5.6.]{Triebel2010}.  
For the assertion in the case $p=1$ we refer to \cite[2.5.7.(2)]{Triebel2010}. 
\item See \cite[\S 2.3 and Thm.~4.2.4]{Triebel1978}.
\item See \cite[Prop.~3.3.1]{Triebel2010}.
\item The inclusion for $\Omega$ replaced by $\mathbb{R}^d$ can be found in 
 \cite[eq.~(2.3.2/5)]{Triebel2010}. Using the definition of the $\Bspace spq$ spaces, 
 this easily implies the assertion.  
\end{enumerate}
\end{proof}

We now recall some well-known results on interpolation of Besov spaces. Besides 
$K$-interpolation reviewed in \Cref{sec:Kinterpol} we also refer to the 
complex interpolation method in some remarks. The latter only works for complex Banach
spaces $X_0,X_1$ as well as some some quasi-Banach spaces, 
and it is denoted by $[X_0,X_1]_{\theta} = X_{\theta}$ for 
$\theta\in (0,1)$ (see \cite{Bergh1976}). 
\begin{proposition}[interpolation of Besov spaces]\label{prop:inter_besov}
Let $s_0,s_1\in \mathbb{R}$, $s_0\neq s_1$, $\theta\in (0,1)$, and 
$s_\theta:= (1-\theta)s_0+ \theta s_1$.
\begin{enumerate}
\item For $p,q_0,q_1\in (0,\infty]$ and $q\in [1,\infty]$ we have 
\begin{align}\label{eq:besov_interp_p_fixed}
\left(\Bspace {s_0}{p}{q_0}, \Bspace {s_1}{p}{q_1} \right)_{\theta,q} 
= \Bspace {s_\theta}{p}{q}.
\end{align}
\item If $p_0,p_1, p_{\theta}\in (0,\infty)$ 
with 
$\frac{1}{p_{\theta}} = \frac{1-\theta}{p_0}+\frac{\theta}{p_1}$, then
\begin{align}\label{eq:besov_interp_p_eq_q}
\left(\Bspace {s_0}{p_0}{p_0}, \Bspace {s_1}{p_1}{p_1} \right)_{\theta,p_{\theta}} 
= \Bspace {s_\theta}{p_\theta}{p_{\theta}}.
\end{align}
\item If  $p_0,p_1, p_{\theta}\in [1,\infty)$ with 
$\frac{1}{p_{\theta}} = \frac{1-\theta}{p_0}+\frac{\theta}{p_1}$ and 
$q_0,q_{\theta}, q_1\in  [1,\infty)$,  with 
$\frac{1}{q_{\theta}} = \frac{1-\theta}{q_0}+\frac{\theta}{q_1}$, then 
\begin{align}\label{eq:complex_besov_interp}
\left[\Bspace {s_0}{p_0}{q_0}, \Bspace {s_1}{p_1}{q_1} \right]_{\theta} 
= \Bspace {s_\theta}{p_\theta}{q_{\theta}}.
\end{align}
\end{enumerate}
\end{proposition}

\begin{proof}
See \cite[Thm.~3.3.6]{Triebel2010} for the first and last statement and 
\cite[Thm.~2.4.3 and Remark 8 in \S 3.3.6]{Triebel2010} for the second statement. 
\end{proof}

\section{On spaces of functions of bounded variation}
Finally, we also recall and generalize some results on functions of bounded variation. 
\begin{proposition}[Embedding]\label{app:bv_embed}
Let $a\geq 0$ with $a\geq \frac{d}{2}-1$. Then there is a continuous embedding $\BV (\Omega)\subset \Bspace {-a}22$. 
\end{proposition}
\begin{proof}
For all embeddings involving the space $\BV(\Omega)$ in this proof we refer to \cite[Cor. 3.49  \& Prop. 3.21]{ambrosio2000functions}.
For $d=1$ there is a continuous embedding $\BV (\Omega) \subset L^2(\Omega)$. We have $L^2(\Omega)=\Bspace 022\subset \Bspace{-a}22$, which yields the claim in this case.\\
For $d>1$ we set $p:=\frac{d}{d-1}$. Then $p\in (1,2]$ and there is a continuous embedding $\BV(\Omega)\subset L^{p}(\Omega)$. By Proposition \ref{app:relation_lp_sob} we have a continuous embedding $L^{p}(\Omega)\subset \Bspace 0p2$. Furthermore, $ a+\frac{d}{2}\geq d-1=\frac{d}{p}$ yields a continuous embedding $\Bspace 0p2 \subset \Bspace {-a}22$. Putting together the latter three embeddings yields the claim.
\end{proof}

\begin{proposition}\label{app:bv_interpolation}
Let $a\geq 0$, $s\in (-a,1)$, and $\Omega\subset \mathbb{R}^d$ a bounded Lipschitz domain.  Then 
\begin{align*}
\Bspace s{t_s}{t_s} = \left( \Bspace {-a}22, \BV(\Omega) \right)_{\theta_s,t_s} \quad\text{ with }  \theta_s:= \frac{s+a}{a+1} \text { and }  t_s:= \frac{2a+2}{s+2a+1}
\end{align*}
with equivalent norms.
\end{proposition}  
\begin{proof}
First note that if $f\in \BV(\mathbb{R}^d)$, then \[ f|_{\Omega}\in \BV(\Omega)\quad \text{with}\quad \|f|_{\Omega}\|_{\BV(\Omega)}\leq \|f\|_{\BV(\mathbb{R}^d)} .\]
Due to \cite[Thm.~ 1.4]{cohen2003harmonic} to claim holds true for $\Omega=\mathbb{R}^d$. Note that here the condition $\gamma <1-\frac{1}{d}$ from the latter reference on $\gamma:=\frac{-(2a+2)}{d}+1$ is satisfied. 
Let $c_1$ be a constant such that the norm in $\Brspace s{t_s}{t_s}$ is bounded by $c_1$ times the norm in $\left( \Brspace {-a}22, \BV(\mathbb{R}^d) \right)_{\theta_s,t_s}$ and the other way around.\\ 
We transfer this result to bounded Lipschitz domains. To this end we  separately prove both inclusions in the stated identity.\\
Let $f\in \Bspace s{t_s}{t_s}$. Then there exists $\tilde{f}\in  \Brspace s{t_s}{t_s}$ with 
\[ \tilde{f}|_{\Omega}=f \quad \text{and} \quad \|\tilde{f}\|_{\Brspace s{t_s}{t_s}} \leq 2 \| f \|_{\Bspace s{t_s}{t_s}}.
\] 
Let $t>0$ and $\tilde{f}=\tilde{f}_1+\tilde{f}_2$ with $\tilde{f}_1\in \Brspace {-a}22$ and $\tilde{f}_2\in \BV(\mathbb{R}^d)$ be a decomposition such that  
\[ \|\tilde{f}_1\|_{\Bospace {-a}22(\mathbb{R}^d)} + t \|\tilde{f}_2\|_{\BV(\mathbb{R}^d)} \leq 2 K(t,\tilde{f}) \] 
with the $K$-functional from real interpolation of Banach spaces. Then $\tilde{f}_1|_\Omega\in \Bspace{-a}22$, $\tilde{f}_2|_{\Omega} \in \BV(\Omega)$,  $f=f_1+f_2$ and 
\[ 
K(t,f)\leq \|\tilde{f}_1|_\Omega\|_{\Bspace {-a}22} + t \|\tilde{f}_2|_{\Omega}\|_{\BV(\Omega)}\leq 2 K(t,\tilde{f}). \] 
Hence with the definition of the norm on real interpolation spaces we obtain 
\[ 
\|f\|_{\left( \Bspace {-a}22, \BV(\Omega) \right)_{\theta_s,t_s}} 
\!\!\!\leq 2  \|\tilde{f}\|_{ \left( \Brspace {-a}22, \BV(\mathbb{R}^d) \right)_{\theta_s,t_s}} 
\!\!\!\leq 2c_1 \| \tilde{f}\|_{\Brspace s{t_s}{t_s}} 
\leq 4c_1 \| f\|_{\Bspace s{t_s}{t_s}} .\]
We turn to the other inclusion.
There exists a constant $C_{\mathrm{ext}}>0$ such that for every ${f\in \Bspace {-a}22}$ there exists $\tilde{f}\in \Brspace{-a}22$ with $\tilde{f}|_{\Omega}=f$ and 
$\|\tilde{f}\|_{\Brspace {-a}22} \leq C_{\mathrm{ext}} \|f\|_{\Bspace {-a}22}$ 
and likewise for every $f\in \BV (\Omega)$ there exists $\tilde{f}\in \BV(\mathbb{R}^d)$ with 
\[ 
\tilde{f}|_{\Omega}=f \quad \text{and}\quad  \|\tilde{f}\|_{\BV(\mathbb{R}^d)} \leq C_{\mathrm{ext}} \|f\|_{\BV(\Omega)}.
\] 
This holds true by the definition of $\Bspace {-a}22$ via restrictions and due to \cite[Prop.~3.21]{ambrosio2000functions} for of bounded variation functions. 
Now suppose $f\in \left( \Bspace {-a}22, \BV(\Omega) \right)_{\theta_s,t_s}$. Let $f=f_1+f_2$ with $f_1\in \Bspace {-a}22 $ and $f_2\in \BV(\Omega)$ such that 
 \[ \|f_1\|_{\Bspace {-a}22} + t \|f_2\|_{\BV(\Omega)} \leq 2 K(t,f). 
 \]
Let $\tilde{f}_1\in \Brspace {-a}22$ and $\tilde{f}_2\in \BV(\mathbb{R}^d)$ be extensions as above. Then $\tilde{f}:=\tilde{f}_1+\tilde{f}_2$ satisfies $\tilde{f}|_{\Omega}=f$, and
\[ 
K(t,\tilde{f}_1+\tilde{f}_2)
\leq \|\tilde{f}_1\|_{\Brspace{-a}22} + t \|\tilde{f}_2\|_{\BV(\mathbb{R}^d)}\leq 2C_{e}   K(t,f).\]
We conclude that  
\begin{align*}
\|f\|_{\Bspace s{t_s}{t_s}} 
&\leq \|\tilde{f}\|_{\Brspace s{t_s}{t_s}} 
\leq c_1 \|\tilde{f}\|_{ \left( \Brspace {-a}22, \BV(\mathbb{R}^d) \right)_{\theta_s,t_s}}  \\
&\leq 2 c_1 C_{e}  \|f\|_{\left( \Bspace {-a}22, \BV(\Omega) \right)_{\theta_s,t_s}}
.
\end{align*}
\end{proof}

\bibliographystyle{abbrv}  
\bibliography{lit} 	
\end{document}